\documentclass[12pt]{article}

\RequirePackage[italian,english]{babel}

\title{Euler-Lagrange Equations for Multiobjective Calculus of Variations Problems  
via Set Optimization}

\author{D.~Visetti\footnote{Corresponding author, Free University of Bozen-Bolzano, 
  Faculty of Economics and Management, 
  \href{mailto:daniela.visetti@unibz.it}{daniela.visetti@unibz.it}.  
  The work of D. Visetti was supported by the project \emph{Verification 
  Techniques for Multicriteria Variational Problems}, Free University of Bolzano-Bozen.}
  \and F.~Heyde\footnote{Freiberg University of Mining and Technology,  
  \href{mailto:frank.heyde@math.tu-freiberg.de}{frank.heyde@math.tu-freiberg.de}}
  }

\RequirePackage[italian,english]{babel} 
\RequirePackage{graphicx}
\RequirePackage{amsmath}
\RequirePackage{amssymb}
\RequirePackage{amsthm}
\usepackage[colorlinks=true, linkcolor=blue, citecolor=blue, urlcolor=blue, filecolor=blue]{hyperref}

\newcommand{\cl}{{\rm cl \,}}
\newcommand{\co}{{\rm co \,}}
\newcommand{\bs}{\backslash}
\renewcommand{\P}{\ensuremath{\mathcal{P}}}
\newcommand{\F}{\ensuremath{\mathcal{F}}}
\newcommand{\G}{\ensuremath{\mathcal{G}}}
\newcommand{\NN}{{\mathbb N}}
\newcommand{\RR}{{\mathbb R}}

\newtheorem{theorem}{Theorem}[section]
\newtheorem{proposition}[theorem]{Proposition}
\newtheorem{lemma}[theorem]{Lemma}
\newtheorem{corollary}[theorem]{Corollary}
\newtheorem{remark}[theorem]{Remark}
\newtheorem{example}[theorem]{Example}
\newtheorem{definition}[theorem]{Definition}

\usepackage{color}

\definecolor{color0}{gray}{.50}
\definecolor{color1}{rgb}{0,.2,.8}
\definecolor{color2}{rgb}{1,.2,0}
\definecolor{color3}{rgb}{.8,.5,1}

\begin{document}
\selectlanguage{english}

\maketitle

\begin{abstract}
  The problem of minimizing an integral functional of a vector-valued Lagrangian 
  on a set of admissible arcs with given endpoints is considered.  The problem is 
  tackled by embedding it into a set-optimization problem such that the image 
  space becomes a complete lattice. This procedure allows to define the 
  concepts of minimizer and of infimizer, as two completely different notions.  
  An infimizer is proved to contain minimizers or at least minimizing sequences 
  for the linearly scalarized problems, with the scalarizing parameters being  
  elements in the dual of the ordering cone.  In this way 
  set-valued Euler-Lagrange equations are obtained and weak solutions are defined 
  for these equations.  Following the guidelines of the classical results, under 
  suitable convexity and coercivity hypotheses an existence result for an infimizer is 
  proved.  In addition to the unconstrained problem, Euler-Lagrange equations are 
  provided for the case with isoperimetric constraints.   Finally, the results are 
  applied to the optimization of the shape of energy-saving buildings.
\end{abstract}



%%%%%%%%%%%%%%%%%%%%%%%%%%%%%%%%%%%%%%
\section{Introduction}

The problem of minimizing a vector-valued functional on a set of admissible arcs 
with given endpoints is a multicriteria calculus of variations problem. Similar 
problems of optimal control type have been called multi-index optimization 
problems and already studied in the 1960s and 70s, see for example \cite{rc}, \cite{yl}. Solving such a  problem is usually understood as looking for a minimal element 
in the set of the images with respect to a vector order generated by a convex cone $C$. The standard 
example is when $C$ is the positive orthant, which generates the component-wise 
order. However, this point of view does not lead to satisfactory results for example 
when it comes to value functions or duality since these concepts involve infima and suprema which may not exist with respect to vector orders or, even if they do, they 
do not provide useful information. Therefore, a new approach is proposed which 
consists of two steps: first, extend the problem to one with values in a complete 
lattice of sets and secondly, apply concepts and tools from set optimization theory as surveyed in \cite{SetOptSurvey}. In particular, the set optimization solution concept is applied here which has two features: on the one hand, one looks for sets which 
generate the infimum of the set extension of the functional (called infimizers), on the other hand, one looks for minimal elements with respect to a set relation. This is a new approach: the goal of this paper are Euler-Lagrange equations characterizing infimizers.

The underlying order relation will be generated by a closed convex cone which can 
be different from the positive orthant. This vector order will then be used to construct the image spaces of the set-valued extension of the problem. This procedure and basic concepts from the theory of ordered vector spaces and the complete-lattice approach to set optimization will be recalled in the next section.

Multi-objective optimization has many applications in engineering, see, e.g., \cite{MarlerArora} and the bibliography therein. For example, it 
is employed in the design, planning and operation of chemical process industries (see \cite{Rangaiah}).  More references 
about the use of multiple objectives in engineering system design and reliability optimization can be found in \cite{KCS}: they range from a residual heat removal safety system for a boiling water reactor, to the selection of technical specifications and maintenance activities at nuclear power plants to increase reliability, availability and maintainability of safety-related equipment or to minimizing total wire length and minimizing the failure rate in the printed circuit board design.

Infimizers are sets, but it is easier to deal with points.  If one passes from the 
set-valued function $f$ to be minimized to the inf-translation 
of $f$ by $M$, the infimizer, which was a set, becomes a set formed by a single 
point: $\{0\}$.    
Using this fact and the scalarizations with respect to elements of the dual cone, 
we prove that an infimizer $M$ must contain the minimizers of the scalarized 
problems or at least their minimizing sequences (in the considered direction).  
For each direction $\zeta$, set-valued Euler-Lagrange equations are derived.  It is 
also possible to write an equation of sets, that contains the information of all the 
scalarizations corresponding to the dual cone.

Like in the well-known real-valued results, under hypotheses of coercivity and 
convexity an existence result is proved:  the set formed by all the solutions of the scalarized problems is an infimizer.  In order to obtain the existence result we 
must enlarge the set of the admissible arcs from the space $C^1$ to a 
Sobolev space.  As a consequence, a concept of weak solutions for set-valued 
Euler-Lagrange equations is proposed.

Since constrained problems are important for applications, we 
considered the isoperimetric problem.  Also here an infimizer $M$ must 
contain the minimizers of the scalarized problems or at least their minimizing 
sequences. The set-valued Euler-Lagrange equations are provided.

Finally, an application to optimization of the shape of energy-saving buildings, 
taken from \cite{Marks}, is studied.

%%%%%%%%%%%%%%%%%%%%%%%%%%%%%%%%%%%%%%
\section{Complete Lattices of Sets: Basic Concepts}

In this section, basic concepts and constructions for the complete-lattice approach to vector and set optimization are recalled. The reader may compare 
\cite{SetOptSurvey} for more details and references.

A nonempty closed convex cone $\emptyset \neq C \subsetneq \RR^d$ generates a vector preorder on $\RR^d$ (i.e., a reflexive and transitive relation which is compatible with addition and multiplication by non-negative numbers) by means of
\[
y \leq_C z \quad \Leftrightarrow \quad z - y \in C
\]
for $y, z \in \RR^d$.  The (positive) dual of the cone $C$ is
\[
C^+ = \{\zeta \in \RR^d \mid \forall z \in C\, ,\ \zeta \cdot z \geq 0\}.
\]
The bipolar theorem yields $C = C^{++} := (C^+)^+$ under the given assumptions.

Let $\P(\RR^d)$ denote the power set of $\RR^d$. On $\P(\RR^d)$ the Minkowski 
sum of two non-empty sets $A,B$ is defined as $A+B=\{a+b\mid a\in A,\; b\in B\}$.  
It can be extended to the whole power set by $A+\emptyset = \emptyset + A= \emptyset$. Furthermore, we consider the sum
\[
A \oplus B = \mbox{cl} \, (A+B)\, ,
\]
where $\mbox{cl}\, (\cdot)$ denotes the closure with respect to the usual topology on $\RR^d$.

The two sets 
\begin{align*}
\F(\RR^d, C) & =\{A \in \P(\RR^d) \mid A = A \oplus C\},\\
\G(\RR^d, C) & = \{A \in \P(\RR^d) \mid A = \mbox{cl\,co}\,(A + C)\}
\end{align*}
comprise the closed and the closed convex subsets of $\RR^d$, respectively, which are stable under addition of $C$. Clearly, $\G(\RR^d, C) \subseteq \F(\RR^d, C)$. The pairs $(\F(\RR^d, C),\supseteq)$, $(\G(\RR^d, C), \supseteq)$ are complete lattices with the following formulas for infimum and supremum (see, for example, \cite{SetOptSurvey}) for a collection $\mathcal{A}
\subseteq \F(\RR^d, C), \G(\RR^d, C)$:
\begin{align}
\label{inf_sup_F}
\inf\mathcal A  = \mbox{cl}\bigcup_{A\in\mathcal{A}}A, \quad & \sup\mathcal A = \bigcap_{A \in \mathcal A} A, \\
\label{inf_sup_G}
\inf \mathcal A  = \mbox{cl\,co}\bigcup_{A \in \mathcal A} A, \quad & \sup \mathcal A = \bigcap_{A \in \mathcal A} A,
\end{align}
respectively. Moreover, $\F(\RR^d, C)$ as well as $\G(\RR^d, C)$ are ordered conlinear spaces in the sense that their order $\supseteq$ is compatible with the addition $\oplus$ and multiplication with non-negative scalars defined by $s A = \{s a \mid a \in A\}$ for $A \neq \emptyset$, $s \emptyset = \emptyset$ whenever $s > 0$, and $0 A = C$ for all $A \in \F(\RR^d, C)$, $\G(\RR^d, C)$. Finally, $\G(\RR^d, C)$ with these operations is a conlinear subspace of $\F(\RR^d, C)$.  For references on conlinear spaces see 
\cite{SetOptSurvey}.
  
For $\zeta \in \RR^d\backslash\{0\}$, we denote
\[
H^+(\zeta) = \{z \in \RR^d \mid \zeta \cdot z\geq 0\}\, .
\]
The $\zeta$-difference for $\zeta \in C^+\bs\{0\}$ of two sets $A, B \in \F(\RR^d,C)$ is the set 
\[
A-_\zeta B=\{z\in\RR^d\mid z+B\subseteq A\oplus H^+(\zeta)\}
\]
which can also be written as
\begin{equation}
\label{minuszeta}
A-_\zeta B=\{z\in\RR^d\mid \zeta\cdot z+\inf_{b\in B}\zeta\cdot b\geq 
  \inf_{a\in A}\zeta\cdot a\}
\end{equation}
(see Remark 2.1 in \cite{HV}).  The $\zeta$-difference is always an element of 
$\G(\RR^d,C)$ and does not change if $A$ and $B$ are replaced by their closed 
convex hulls.  It is a closed half-space or the empty set or all the space.

Let $\{ A_n\}_{n\in\NN}$ be a sequence of sets in $\F(\RR^d,C)$, we 
denote by $\lim_{n\to\infty}A_n$ the following set:
$$
\lim_{n\to\infty}A_n=\left\{ z\in\RR^d\ \vert\ \forall n\in\NN, \exists
  z_n\in A_n, \lim_{n\to\infty}z_n=z\right\}.
$$
This definition of limit coincides with the upper limit of Painlev\'e-Kuratowski 
(see \cite{Aubin_Frank}, $\mathrm{Liminf}_{n\to\infty}A_n=\left\{ z\in Z\ 
\vert\ \lim_{n\to\infty}d(z,A_n)=0\right\}$).

Let $\{ A_s\}_{s\in S}$ with $S\subseteq\RR$, $S=[s_1,s_0)$ or $S=(s_0,s_2]$, 
$s_0,s_1,s_2\in \RR$, be a family of sets in $\mathcal{F}(\RR^d,C)$. We 
denote by $\lim_{s \to s_0}A_s$ the set which satisfies that for any sequence 
$\{s_n\}_{n\in\NN}\subseteq S$ with $s_n \to s_0$ one has
$\lim_{s \to s_0}A_s = \lim_{n\to\infty}A_{s_n}$.

Let $X$ be a vector space and $f \colon X\to \F(\RR^d, C)$ be a 
set-valued function. The graph of $f$ is the set
\[
\mbox{graph}\, f=\{(x,z)\in X\times\RR^d\mid z\in f(x)\} \subseteq X\times \RR^d
\]
and its (effective) domain is the set
\[
\mbox{dom}\, f=\{ x\in X\mid f(x)\neq\emptyset\}\subseteq X\, .
\]
The function $f$ is called convex if and only if $\mbox{graph}\, f$ is a convex set.  
This is equivalent to the following inclusion:  for any $x_1, x_2 \in X$ and 
$\lambda \in(0,1)$
\[
f(\lambda x_1+(1-\lambda)x_2) \supseteq \lambda f(x_1)+(1-\lambda)f(x_2)\, .
\]
A convex function $f \colon X\to \F(\RR^d, C)$ automatically maps into 
$\G(\RR^d, C)$ , i.e., it has convex values.

For $M \subseteq X$, we set
\[
f[M]=\{f(x) \mid x \in M\} \, .
\]

A derivative concept that is most compatible with the complete lattice approach and also the solution concept introduced in the next section is as follows.

Let $\zeta \in C^+\backslash\{0\}$ and $v \in X$. The directional derivative with 
respect to $\zeta$ of $f \colon X\to \F(\RR^d, C)$ at $x_0\in X$ in direction 
$v$ is
\begin{equation}
\label{derivative}
f'_\zeta(x_0)(v)=\lim_{s\to 0^+}\frac{1}{s} \left( f(x_0+sv)-_\zeta f(x_0) \right)
\end{equation}
(see \cite{HS} and \cite{HV} for a motivation and many features including a 
discussion of the ``improper" function values $\RR^d$ and $\emptyset$).   

We can also consider the derivative of $f$ along a curve $\gamma:
[s_1,s_2]\to X$ with respect to $\zeta$ at $s_0\in[s_1,s_2]$:
\begin{equation}
\label{derivative_curve}
\left.\frac{d\ }{ds}f_\zeta(\gamma(s))\right|_{s=s_0}=\lim_{s\to s_0}\frac{1}{s-s_0}
  \left( f(\gamma(s))-_\zeta f(\gamma(s_0)) \right).
\end{equation}

In the set-valued theory, a special type of function replaces linear functionals.
Let $X$ be a topological vector space and $X'$ its dual space.  For $\xi \in X'$ 
and $\zeta \in C^+\backslash\{0\}$, the function $S_{(\xi,\zeta)} \colon X
\to \G(\RR^d,C)$ is defined by
\begin{equation}
\label{S}
S_{(\xi,\zeta)}(x) = \{z \in \RR^d \mid \zeta \cdot z \geq \xi(x)\}\, .
\end{equation}
For each $\xi \in X'$ and $\zeta \in C^+\backslash\{0\}$, $S_{(\xi,\zeta)}$ is positively homogeneous and additive, hence in particular convex (see \cite{SetOptSurvey}).

%%%%%%%%%%%%%%%%%%%%%%%%%%%%%%%%%%%%%%
\section{Set-valued Euler-Lagrange Equation}

Let $a, b \in \RR$ be two real numbers with $a < b$, $n, d$ two positive integers 
and $L \colon [a,b] \times\RR^{2n} \to \RR^d$, where $d$ is greater than 1, be 
a function of class $C^1$.  We consider the functional
\begin{equation}
\label{J}
J(y) = \int_a^b L(t,y(t),\dot y(t))\, dt
\end{equation}
on a set $\mathcal{X}$ of admissible arcs, taking values in $\RR^n$ and with 
given endpoints:
\begin{equation}
\label{X}
\mathcal{X}= \{ y\in C^1([a,b];\RR^n)\mid y(a)=A,\ y(b)=B\}\, ,
\end{equation}
for $A, B \in \RR^n$.

In this section, we first clarify in what sense the problem of minimizing the 
functional \eqref{J} on the set $\mathcal{X}$, shall be solved. The set-valued 
extension of the problem is established by extending $L$ and $J$ to functions into 
$\F(\RR^d, C)$. The function 
$\overline L \colon [a,b] \times \RR^{2n} \to \F(\RR^d,C)$ is simply defined by
\[
\overline L(t,y,p) = L(t,y,p) + C
\]
and the extension of $J$ is $\overline J \colon \mathcal{X} \to \F(\RR^d, C) $ defined by
\[
\overline J(y) = \int_a^b \overline L(t,y(t),\dot y(t))\, dt
\]
where the integral is understood in the Aumann sense (see \cite{Aumann}), 
but one may just note that $\overline J(y) = J(y) + C$. Clearly, since $L$ and $J$ are vector-valued and $C$ is closed and convex, 
the two extensions actually map into $\G(\RR^d, C)$. However, they will be considered 
as $\G(\RR^d, C)$-valued only if convexity is assumed, 
otherwise as $\F(\RR^d, C)$-valued.  The choice influences the formula for the 
infimum as one may see from \eqref{inf_sup_F}, \eqref{inf_sup_G}.

In this paper, the following solution concept for the problem
\[
\tag{P} \text{minimize} \quad \overline J(y) \quad \text{subject to} \quad y \in 
\mathcal X
\]
is used. It is an adaption of the concept from \cite{HL} which is, for more abstract problems, due to \cite{HS}.

\begin{definition}
\label{SolConcept}
A set $M \subseteq \mathcal X$ is called an infimizer of (P) if
\begin{equation}
\label{Infimizer}
\inf \overline J[M] = \inf \overline J[\mathcal X].
\end{equation}

An element $y_\zeta \in \mathcal X$ is called a $\zeta$-solution of (P) with 
$\zeta \in C^+\backslash\{0\}$ if
\begin{equation}
\label{ZetaSol}
\zeta \cdot J(y_\zeta) = \inf\{\zeta \cdot \overline J(y) \mid y \in \mathcal X\}.
\end{equation}
An infimizer $M \subseteq \mathcal X$ of (P) which only includes $\zeta$-solutions 
with $\zeta \in C^+\backslash\{0\}$ is called a $C^+$-solution of (P).
\end{definition}

It is immediate that $M = \mathcal X$ is an infimizer, and that if $M$ is 
an infimizer of (P), then so is any other set $M' \subseteq \mathcal X$ with $M' 
\supseteq M$. Since an infimizer is a subset of $\mathcal X$, it may seem hard to 
formulate necessary conditions for a set to be an infimizer. The following concept 
reduces infimizers to singletons. It was introduced in \cite{HS}.

\begin{definition}
\label{InfTranslation}
Let $M\subseteq\mathcal{X}$ be a non-empty set of feasible arcs. The function 
$$
\widehat J(\cdot;M) \colon C^1_0([a,b]; \RR^n) \to \F(\RR^d, C)
$$
defined by
\begin{equation}
\label{inf_transl}
\widehat J(y; M) = \inf_{v \in M} \overline J(y+v)
\end{equation}
is called the inf-translation of $\overline J$ by $M$.
\end{definition}

The infimum in \eqref{inf_transl} is understood in $(\F(\RR^d, C), \supseteq)$. Note that $\widehat J(\cdot; M)$ is not convex in general even if $\overline J$ is. On the other hand, it is easy to see that a set $M$ is an infimizer for $\overline J$ if and only if $\{0\} \subset C_0^1([a,b]; \RR^n)$ is an infimizer for $\widehat J(\cdot;M)$ (see \cite{SetOptSurvey}).   In \cite{inf-transl} the link between optimization problems 
and the inf-translation is studied.

It is important to realize that even though the original problem is vector-valued, its inf-translation
\[
\tag{IP} \text{minimize} \quad\widehat J(y; M) \quad \text{subject to} \quad y \in C^1_0([a,b]; \RR^n)
\]
is a genuine set-valued problem.

\begin{remark}
  Since the boundary conditions are essentially linear constraints, the inf-translation 
  can be considered as a function on the linear space $C_0^1([a,b]; \RR^n)$ which 
  ``absorbs" the boundary conditions, since for $v \in \mathcal X$ and $y \in 
  C_0^1([a,b]; \RR^n)$ one has $y + v \in \mathcal X$.
\end{remark}

In \eqref{ZetaSol}, linear scalarizations of the vector-valued problems were used. This is extended to $\widehat J$ by considering the function $\varphi_{\zeta, M} \colon 
C^1_0([a,b]; \RR^n) \to \overline\RR:=\RR\cup\{\pm\infty\}$ defined by
\begin{equation}
\label{phi}
\varphi_{\zeta, M}(v) = \inf_{z\in\widehat J(v; M)} \zeta\cdot z
\end{equation}
with $\zeta \in C^+$.

Since the following equalities hold
$$
\begin{aligned}
\varphi_{\zeta,M}(v) & =\inf_{z\in\widehat J(v;M)} \zeta\cdot z 
    = \inf\left\{ \zeta\cdot z\mid z\in \inf_{y\in M}\overline J(y+v)= 
    \mbox{cl}\bigcup_{y\in M}\overline J(y+v)\right\} \\
  &= \inf\left\{ \zeta\cdot z\mid z\in\bigcup_{y\in M} J(y+v)\right\},
\end{aligned}
$$
one has
\begin{equation}
\label{phi2}
\varphi_{\zeta,M}(v)=\inf_{y\in M} \zeta\cdot J(y+v)\, .
\end{equation}

The relation between an infimizer of the set-valued function $\overline J$ and the 
corresponding extended real-valued function $\varphi_{\zeta,M}$ is explained in the following lemma.

\begin{lemma}
\label{lmm_infM0}
  If $M \subset \mathcal{X}$ is an infimizer of $\overline J$, then $0 \in 
  C^1_0([a,b]; \RR^n)$ is a minimizer of $\varphi_{\zeta, M}$ for every $\zeta \in 
  C^+\backslash\{0\}$.

  If $\bar J$ is convex as a function into $\F(\RR^d, C)$ and $M$ is a convex set, 
  then $M$ is an infimizer if and only if $0 \in C^1_0([a,b]; \RR^n)$ is a minimizer 
  of $\varphi_{\zeta, M}$ for every $\zeta\in C^+\backslash\{0\}$.
\end{lemma}

\begin{proof}
  If $M\subset\mathcal{X}$ is an infimizer for $\overline J$, using (\ref{phi}), 
  we have for all $v\in C_0^1[a,b]$
  $$
  \begin{aligned}
  \varphi_{\zeta,M}(v) &= \inf\left\{\zeta\cdot z\mid z\in\widehat J(v;M)\right\} 
      = \inf\left\{\zeta\cdot z\mid z\in\inf_{y\in M}\overline J(y+v)\right\} \\
    &\geq\inf\left\{\zeta\cdot z\mid z\in\inf_{y\in\mathcal{X}}\overline J(y)\right\} 
      =\inf\left\{\zeta\cdot z\mid z\in\inf_{y\in M}\overline J(y)\right\} 
      =\varphi_{\zeta,M}(0)\, .
  \end{aligned}
  $$

  To prove the second part of the statement, let us suppose that $M$ is not an 
  infimizer. Then there exists $y_*\in\mathcal{X}$ such that 
  \[
  J(y_*)\not\in \inf_{y\in M}\overline{J}(y)=:\Gamma_M,
  \]
  since otherwise $\bigcup_{y\in \mathcal{X}}\{J(y)\} \subseteq\inf_{y\in M}
  \overline{J}(y)$ and consequently
  \[
  \inf_{y\in X} \overline{J}(y)=\cl\left(\bigcup_{y\in \mathcal{X}}\{ J(y)\}+C\right)
  \subseteq \cl\left(\inf_{y\in M}\overline{J}(y)+C\right)=\inf_{y\in M}\overline{J}(y)
  \]
  thus $M$ would be an infimizer.

  Since $\Gamma_M$ is nonempty closed and convex, $J(y_*)$ can be strongly 
  seperated form $\Gamma_M$, i.e. there is $\xi$ in $\RR^d\setminus\{ 0\}$ 
  such that 
  \begin{equation}\label{eq:sep}
  \xi\cdot J(y_*) < \inf_{z\in \Gamma_M} \xi \cdot z.
  \end{equation}
  Since $M$ is nonempty there exists $\hat{y}\in M$. For all $c\in C$, $J(\hat y)
  +c\in \Gamma_M$. We can conclude that $\xi \in C^+$, since otherwise $\xi 
  \cdot \hat c<0$ for some $\hat c\in C$ and consequently
  \[
  \inf_{z\in \Gamma_M} \xi \cdot z \le \inf_{t>0} \xi \cdot \left(J(\hat y)+t\hat c\right)
    = \xi \cdot J(\hat y)+\inf_{t>0} t\left(\xi \cdot \hat c\right) =-\infty,
  \]
  contradicting \eqref{eq:sep}.

  With $v:=y_*-\hat y \in C_0^1([a,b];\RR^n)$ we obtain from \eqref{eq:sep}:
  \[
  \varphi_{\xi,M}(v)\le \xi\cdot J(v+\hat y)=\xi\cdot J(y_*)< \inf_{z\in \Gamma_M} 
    \xi   \cdot z =\varphi_{\xi,M}(0),
  \]
  a contradiction to the assumptions, which completes the proof.
\end{proof}

To simplify the notation, we denote
$$
L_\zeta(t,y,p)=L(t,y,p)\cdot\zeta\, ,
$$
while $D_yL_\zeta(t,y,p)$, $D_pL_\zeta(t,y,p)$ are the vectors in $\RR^n$ of the 
derivatives with respect to $y_i$ and with respect to 
$p_i$ for $i=1,2,\dots,n$.   Analogously, we write 
$$
J_\zeta(y)=\zeta\cdot J(y)=\int_a^b L_\zeta(t,y,\dot y)dt\, .
$$

\begin{remark}
  If $M$ is an infimizer for $\overline J$, from the definition of infimizer,  for 
  every $\zeta\in C^+\bs\{0\}$, we have
  $$
  \inf_{z\in\inf_{y\in M}\overline J(y)}\zeta\cdot z=
  \inf_{z\in\inf_{y\in\mathcal{X}}\overline J(y)}\zeta\cdot z\, .
  $$
  As in the previous lemma, this equation can be written as
  $$
  \inf_{y\in M}J_\zeta(y)=\inf_{y\in\mathcal{X}}J_\zeta(y)\, .
  $$
  We can deduce that  for every $\zeta\in C^+\bs\{0\}$ $M$ must contain 
  either a $\zeta$-solution, i.e. an arc $y_\zeta$ such that
  \begin{equation}
  \label{yzeta}
  J_\zeta(y_\zeta)=\inf_{y\in\mathcal{X}} J_\zeta(y)
  \end{equation}
  or an infimizing sequence $\{ y_{m,\zeta}\}_{m\in\NN}$ such that
  $\lim_{m\to +\infty}J_\zeta(y_{m,\zeta})=\inf_{y\in\mathcal{X}} J_\zeta(y)$.
\end{remark}

\begin{proposition}
\label{prp-phi}
  Let $L:[a,b]\times\RR^{2n}\to\RR^d$ be of class $C^1$.  Let $M\subset\mathcal{X}$ 
  be an infimizer for $\overline J$.  Given $\zeta\in C^+\bs\{0\}$, suppose that there 
  exists a $\zeta$-solution $y_\zeta\in M$.  
  Then $\varphi_{\zeta,M}\colon C_0^1([a,b];\RR^n)\to\overline\RR$ is 
  Fr\'echet differentiable at $0$ and
  \begin{equation}
  \label{phi'zetaM0}
  \varphi_{\zeta,M}'(0)(v)=\int_a^b\left[D_y L_\zeta(t,y_\zeta,\dot y_\zeta)\cdot v+
    D_p L_\zeta(t,y_\zeta,\dot y_\zeta)\cdot\dot v \right] dt\, .
  \end{equation}
\end{proposition}

\begin{proof}
  For any $v\in C_0^1([a,b];\RR^n)$, the following inequalities hold:
  $$
  \begin{aligned}
  \varphi_{\zeta,M}(v) -\varphi_{\zeta,M}(0) &= \inf_{y\in M} \int_a^b 
      L_\zeta(t,y+v,\dot y+\dot v)dt - \int_a^b L_\zeta(t,y_\zeta,\dot y_\zeta)dt \\
  &\leq \int_a^b L_\zeta(t,y_\zeta+v,\dot y_\zeta+\dot v)dt - \int_a^b 
    L_\zeta(t,y_\zeta,\dot y_\zeta)dt\, .
  \end{aligned}
  $$
  By the Fr\'echet differentiability of the functional $J_\zeta:C^1([a,b];\RR^n)\to
  \RR$ one can conclude that
  $$
  \begin{aligned}
  J_\zeta(y_\zeta+v)- J_\zeta(y_\zeta) &=\int_a^b L_\zeta(t,y_\zeta+v,\dot y_\zeta
      +\dot v)dt - \int_a^b L_\zeta(t,y_\zeta,\dot y_\zeta)dt \\
    &= J_\zeta'(y_\zeta)(v)+o(\| v\|_{C^1})\\
    &= \int_a^b\left[D_y L_\zeta(t,y_\zeta,\dot y_\zeta)\cdot v
      +D_p L_\zeta(t,y_\zeta,\dot y_\zeta)\cdot\dot v \right] dt + 
      o(\| v\|_{C^1})
  \end{aligned}
  $$
  and we have proved that
  \begin{equation}
  \label{Frechet1}
  \varphi_{\zeta,M}(v) -\varphi_{\zeta,M}(0) \leq \int_a^b\left[D_y L_\zeta(t,
    y_\zeta,\dot y_\zeta)\cdot v+D_p L_\zeta(t,y_\zeta,\dot y_\zeta)\cdot\dot v 
    \right] dt + o(\| v\|_{C^1})\, .
  \end{equation}

  On the other hand, by Lemma~\ref{lmm_infM0}, we have that 
  $\varphi_{\zeta,M}(v) -\varphi_{\zeta,M}(0)\geq 0$. By this and (\ref{Frechet1}), 
  the following inequality holds
  $$
  \int_a^b\left[D_y L_\zeta(t,y_\zeta,\dot y_\zeta)\cdot v
      +D_p L_\zeta(t,y_\zeta,\dot y_\zeta)\cdot\dot v \right] dt + 
      o(\| v\|_{C^1})\geq 0\, .
  $$
  Using the same argument with $-v$, one obtains that
  $$
  \int_a^b\left[D_y L_\zeta(t,y_\zeta,\dot y_\zeta)\cdot v
      +D_p L_\zeta(t,y_\zeta,\dot y_\zeta)\cdot\dot v \right] dt + 
      o(\| v\|_{C^1})\leq 0
  $$
  and so
  $$
  \int_a^b\left[D_y L_\zeta(t,y_\zeta,\dot y_\zeta)\cdot v+
    D_p L_\zeta(t,y_\zeta,\dot y_\zeta)\cdot\dot v 
      \right] dt =o(\| v\|_{C^1})\, .
  $$

  Then also the other inequality holds:
  \begin{equation}
  \label{Frechet2}
  \varphi_{\zeta,M}(v) -\varphi_{\zeta,M}(0)\geq \int_a^b\left[D_y 
      L_\zeta(t,y_\zeta,\dot y_\zeta)\cdot v+D_p L_\zeta(t,y_\zeta,\dot y_\zeta)
      \cdot\dot v \right] dt + o(\| v\|_{C^1[a,b]})\, .
  \end{equation}
  Now by (\ref{Frechet1}) and (\ref{Frechet2}) the proof is complete.
\end{proof}

From the fact that $y_\zeta$ is a minimizer of the functional $J_\zeta(y)$ with 
Lagrangian $L_\zeta(t,y,p)$ and the classical theory of calculus of variations (see 
e.g. \cite{Evans} or \cite{Giaq}), in \eqref{phi'zetaM0} we obtain that 
$\varphi_{\zeta,M}'(0)(v)=0$ for any $v\in C_0^1([a,b];\RR^n)$.  Moreover, 
the following corollary holds.

\begin{corollary}
  Let $L:[a,b]\times\RR^{2n}\to\RR^d$ be of class $C^2$.  Let $M\subset
  \mathcal{X}$ 
  be an infimizer for $\overline J$.  Given $\zeta\in C^+\bs\{0\}$, suppose that 
  there exists a $\zeta$-solution $y_\zeta\in M$ of class $C^2$.  
  Then $y_\zeta$ is a solution of the Euler-Lagrange equation in the direction 
  $\zeta$
  \begin{equation}
  \label{EL_zeta}
  \frac{d}{dt}D_p L_\zeta(t,y,\dot y)=D_y L_\zeta(t,y,\dot y)\, .
  \end{equation}
\end{corollary}

\begin{remark}
\label{rmk-|zeta|=1}
  In order to find the solutions $y_\zeta$ of (\ref{EL_zeta}) it is sufficient to consider 
  only some $\zeta\in C^+\bs\{0\}$, because if $\zeta_1=k\zeta_2$, with $\zeta_1,  
  \zeta_2\in C^+\bs\{0\}$ and $k>0$, then $y_{\zeta_1}=y_{\zeta_2}$.  So it is 
  enough to choose only $\zeta$ from a base of $C^+$, that is a set  $B\subseteq 
  C^+\bs\{0\}$ such that for each element $\zeta\in C^+\bs\{0\}$ there are 
  unique $\zeta_0\in B$ and $s > 0$ with $\zeta = s\zeta_0$.  The base $B$ may 
  consist of all $\zeta\in C^+\bs\{0\}$ with 
  unitary norm or, for certain cones, of all $\zeta=(\zeta_1,\dots,\zeta_d)\in C^+\bs
  \{0\}$ such that $\zeta_1+\dots+\zeta_d=1$.
\end{remark}

From the results in Proposition~\ref{prp-phi} on the scalarizations 
$\varphi_{\zeta,M}$, there are consequences on the corresponding inf-translation:  
they are described in the following corollary, recalling that $S_{(\xi,\zeta)}:
C_0^1([a,b];\RR^n)\to\G(\RR^d,C)$ is defined in \eqref{S}.

\begin{corollary}
  Let $L:[a,b]\times\RR^{2n}\to\RR^d$ be of class $C^1$.  Let $M\subset
  \mathcal{X}$ be an infimizer for $\overline J$.  Given $\zeta\in C^+\bs
  \{0\}$, suppose that there exists a $\zeta$-solution $y_\zeta\in M$.  
  Then for any $v\in C_0^1([a,b];\RR^n)$
  \begin{equation}
  \label{setELzeta}
  \widehat J'_\zeta(0;M)(v)=S_{(\varphi'_{\zeta,M}(0),\zeta)}(v)=H^+(\zeta)\, .
  \end{equation}
\end{corollary}

\begin{proof}
  Recalling definition \eqref{derivative}, for $s>0$, using \eqref{minuszeta}, we 
  have
  $$
  \begin{aligned}
  \frac{1}{s} &\left(\widehat J(sv;M)-_\zeta\widehat J(0;M)\right) =\left\{\frac{1}{s}z\in
      \RR^d\mid z\in\widehat J(sv;M)-_\zeta\widehat J(0;M)\right\} \\
    &= \left\{\frac{1}{s}z\in\RR^d\mid\zeta\cdot z+\inf_{z_0\in\widehat J(0;M)}
      \zeta\cdot z_0\geq\inf_{z_1\in\widehat J(sv;M)}\zeta\cdot z_1\right\} \\
    &= \left\{\frac{1}{s}z\in\RR^d\mid\zeta\cdot z+\varphi_{\zeta,M}(0)\geq
      \varphi_{\zeta,M}(sv)\right\} \\
    &= \left\{ z\in\RR^d\mid\zeta\cdot z\geq\frac{1}{s}(\varphi_{\zeta,M}(sv)-
      \varphi_{\zeta,M}(0))\right\}.
  \end{aligned}
  $$
  Taking the limit as $s$ tends to zero and by Proposition~\ref{prp-phi}, we 
  obtain that
  $$
  \widehat J'_\zeta(0;M)(v)= \left\{ z\in\RR^d\mid\zeta\cdot z\geq\varphi'_{\zeta,M}
    (0)(v)\right\}=H^+(\zeta)\, .
  $$
\end{proof}

The previous corollary says that if for some $\zeta\in C^+\bs\{0\}$ there exists 
a $\zeta$-solution $y_\zeta$ and if $M$ is an infimizer for $\overline J$ 
containing it, then the set-valued version of the Euler-Lagrange 
equation is \eqref{setELzeta} (see also \eqref{phi'zetaM0}).  Solving this equation 
is a necessary condition for $M$ being an infimizer.    If the conditions hold for 
every $\zeta\in C^+\bs\{0\}$ a unique equation can be written:
\begin{equation}
\label{setEL}
\sup_{\zeta\in C^+\bs\{0\}} \widehat J'_\zeta(0;M)(v)=C\, .
\end{equation}

\begin{example}
\label{ex_2}
  We want to move a mass point with unitary mass in one dimension from point 
  $A$ at time $t=0$ to point $B$ at time $t=1$ and we want to do that in a way 
  that minimizes the difference of kinetic energy and potential energy over time. 
  The potential energy might be generated by some external forces $f$ (as e.g. 
  gravity) that may vary with time. We assume two scenarios, either there are no 
  external forces, then the Lagrangian would be $L_1(t,y,p)=\frac12p^2$ and 
  the movement should take place with constant velocity, or there is a given 
  external force $f$, then the Lagrangian would be $L_2(t,y,p)=\frac12p^2+
  yf(t)$ and another movement would be optimal. Assume that we don't know 
  which scenario will occur. If we move with constant velocity which is optimal in 
  scenario 1 this might be rather expensive in scenario 2 and vice versa. So the 
  aim would be to find some compromise solution which is pretty good for both 
  scenarios. If there are known probabilities $\zeta_i$ for occurence of scenario 
  $i$, then minimizing the integrated expected energies is the same as using 
  the Lagrangian $L_\zeta$.
  
  The two components $L_1,L_2$ correspond to the Dirichlet principle and the 
  generalized Dirichlet principle, respectively: a solution $y(t)$ of $\ddot y=0$, 
  respectively of $\ddot y=f$, is a minimizer of $\int_0^1L_1(t,y,\dot y)dt$ or of 
  $\int_0^1L_2(t,y,\dot y)dt$.

  Let us consider the following Lagrangian:
  $$
  L(t,y,p)=\left( \begin{array}{c} \frac{1}{2}p^2 \\ \frac{1}{2}p^2+yt 
    \end{array} \right),
  $$
  with $t\in[0,1]$.   For $\zeta=(\zeta_1,\zeta_2)$, we have
  $L_\zeta(t,y,p)=\frac{\zeta_1}{2}p^2+\zeta_2\left(\frac{1}{2}p^2+yt\right)$.   
  The Euler-Lagrange equation is $(\zeta_1+\zeta_2)\ddot y=\zeta_2 t$ 
  and, using the boundary condition in the case $\zeta_1+\zeta_2\neq 0$, the 
  $\zeta$-solution $y_\zeta$ is
  \begin{equation}
  \label{yzeta_ex2}
  y_\zeta(t)=\frac{\zeta_2}{6(\zeta_1+\zeta_2)}t^3+\left(B-A-\frac{\zeta_2}{6(
    \zeta_1+\zeta_2)}\right)t+A\, .
  \end{equation}

  If we write the previous equation only for $\zeta=(\zeta_1,1-\zeta_1)$, with 
  $0\leq\zeta_1\leq 1$, we obtain
  $$
  y_\zeta(t)=\frac{1-\zeta_1}{6}t^3+\left(B-A-\frac{1-\zeta_1}{6}\right)t+A\, .
  $$
  In Figure~\ref{fig1} there are the graphs of the three arcs corresponding to 
  $\zeta=(0,1)$, $\left(\frac{1}{2},\frac{1}{2}\right)$ and $(1,0)$ for the choice 
  $A=0$ and $B=1$.
  \begin{figure}
  \begin{center}
  \includegraphics[height=8cm]{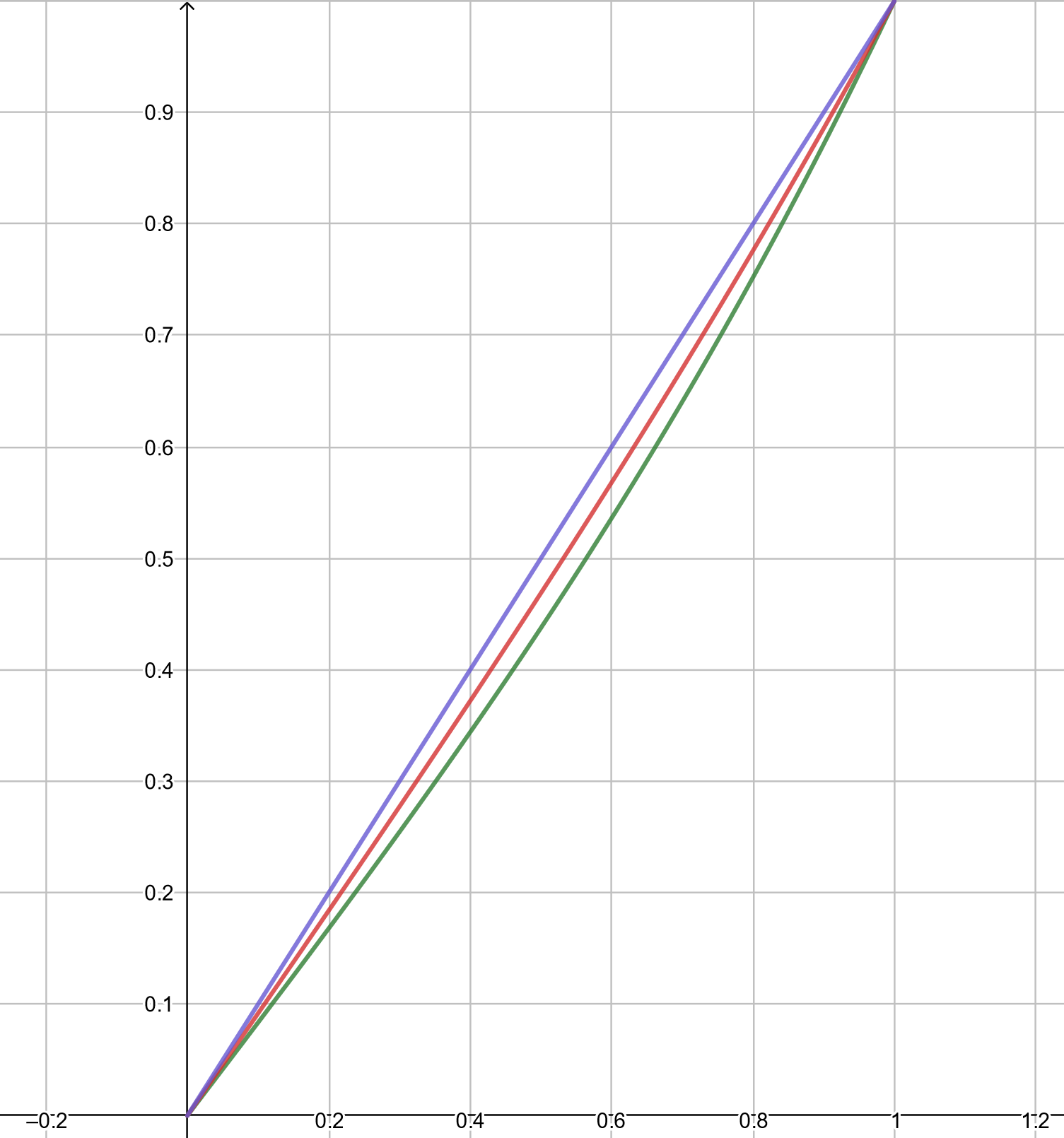}
  \end{center}
  \caption{If we choose $A=0$ and $B=1$, the graphs of the three arcs 
  $y_\zeta$, corresponding to $\zeta=(0,1)$, $\left(\frac{1}{2},\frac{1}{2}\right)$ 
  and $(1,0)$, are represented here. The green graph is $y_{(0,1)}(t)=
  \frac{1}{6}t^3+\frac{5}{6}t$.  The red graph is $y_{\left(\frac{1}{2},
  \frac{1}{2}\right)}(t)=\frac{1}{12}t^3+\frac{11}{12}t$.  Finally, the blue 
  graph is $y_{(1,0)}(t)=t$.}
  \label{fig1}
  \end{figure}
  We can conclude that the set $M$ should contain the arcs $y_\zeta$.
\end{example}

%%%%%%%%%%%%%%%%%%%%%%%%%%%%%%%%%%%%%
\section{Existence of an Infimizer}

In this section we consider a bigger set of admissible arcs, in order to obtain results 
of existence.  More precisely, the arcs need not to be $C^1$ but only to be in a 
suitable Sobolev space $W^{1,q}([a,b];\RR^n)$ with $1<q<+\infty$.   First 
of all we see differences and similarities with respect to the previous arcs' space.

The admissible set is
$$
\mathcal{X}_1=\{ y\in W^{1,q}([a,b];\RR^n)\mid y(a)=A,\ y(b)=B\}\, ,
$$
where we use the fact that if $y\in W^{1,q}([a,b];\RR^n)$, there exists $\tilde y\in 
C([a,b];\RR^n)$ with $y=\tilde y$ almost everywhere.

We assume that all the components of the Lagrangian satisfy the inequality
\begin{equation}
\label{ineq_L}
|L_i(t,y,p)|\leq K_1 \left( |y|^q+|p|^q+1 \right)
\end{equation}
for some constant $K_1$ and all $t\in[a,b]$, $y\in\RR^n$ and $p\in\RR^n$.  

In this section the functional is $\overline J\colon\mathcal{X}_1\to\G(\RR^d,
C)$, defined in the same way as before (condition \eqref{ineq_L} ensures that the 
functional is well defined).  If $M\subseteq\mathcal{X}_1$ is 
non-empty, the inf-translation of $\overline J$ by $M$ is a function
$$
\widehat J(\cdot;M)\colon W_0^{1,q}([a,b];\RR^n)\to\G(\RR^d,C).
$$

In the space $W^{1,q}([a,b];\RR^n)$ one cannot expect to have solutions of the 
Euler-Lagrange equation (\ref{EL_zeta}).  Let us suppose that the Lagrangian is 
$C^1$ and for every component the following 
inequalities hold
\begin{equation}
\label{der_L}
\left\{ \begin{array}{l} |D_pL_i(t,y,p)| \leq K_2 \left( |y|^{q-1}+|p|^{q-1}+1 \right)\\
|D_yL_i(t,y,p)| \leq K_2 \left( |y|^{q-1}+|p|^{q-1}+1 \right)
\end{array} \right.
\end{equation}
for some constant $K_2$ and all $t\in[a,b]$, $y\in\RR^n$ and $p\in\RR^n$.
We say that $y\in\mathcal{X}_1$ is a weak solution of the 
Euler-Lagrange equation in the direction $\zeta$ (\ref{EL_zeta}) provided that
$$
\int_a^b \left[ D_pL_\zeta(t,y,\dot y)\cdot\dot v+D_yL_\zeta(t,y,\dot y)\cdot v 
  \right] ds=0
$$
for all $v\in W_0^{1,q}([a,b];\RR^n)$.

In this section, for any $\zeta\in C^+\bs\{0\}$ and $M\subseteq
\mathcal{X}_1$, the function $\varphi_{\zeta,M}$ is understood as 
$\varphi_{\zeta,M}\colon W_0^{1,q}([a,b];\RR^n)\to\overline\RR$.  In this 
context Proposition \ref{prp-phi} can be reformulated, as it is stated in the following 
proposition.

\begin{proposition}
\label{prp-phi2}
  Let $L:[a,b]\times\RR^{2n}\to\RR^d$ be of class $C^1$ and satisfies the 
  inequalities \eqref{ineq_L} and \eqref{der_L}.  Let $M\subset\mathcal{X}_1$ 
  be an infimizer for $\overline J$.  Given $\zeta\in C^+\bs\{0\}$, suppose that there 
  exists a $\zeta$-solution $y_\zeta\in M$.  
  Then $\varphi_{\zeta,M}\colon W_0^{1,q}([a,b];\RR^n)\to\overline\RR$ is 
  Fr\'echet differentiable at $0$ and
  \begin{equation}
  \label{phi'zetaM0bis}
  \varphi_{\zeta,M}'(0)(v)=\int_a^b\left[D_y L_\zeta(t,y_\zeta,\dot y_\zeta)\cdot v+
    D_p L_\zeta(t,y_\zeta,\dot y_\zeta)\cdot\dot v \right] dt\, .
  \end{equation}
\end{proposition}

Also in this case it is immediate to conclude that $\varphi_{\zeta,M}'(0)(v)=0$ for 
any $v\in W_0^{1,q}([a,b];\RR^n)$.

In order to establish a set-valued Euler-Lagrange equation in the weak sense, 
let $W^{-1,q'}([a,b];\RR^n)$ be the dual space of $W_0^{1,q}([a,b];\RR^n)$, 
where $q'=\frac{q}{q-1}$, and let $\langle\cdot,\cdot\rangle$ denote the natural 
pairing. We write
\begin{equation}
\label{weak_sol}
\langle -\frac{d\ }{dt}D_pL_\zeta(t,y,\dot y)+D_yL_\zeta(t,y,\dot y),v\rangle 
  = \int_a^b \left( D_pL_\zeta(t,y,\dot y)\cdot\dot v+D_yL_\zeta(t,y,\dot y)\cdot v 
  \right) ds\, ,
\end{equation}
where the second derivative of the Lagrangian is meant 
in the weak sense, as it can be seen in the right-hand side.  
For any $\xi\in 
W^{-1,q'}([a,b];\RR^n)$  and $\zeta\in C^+\bs\{0\}$, we can 
consider $S_{(\xi,\zeta)}:W_0^{1,q}([a,b];\RR^n)\to\G(\RR^d,C)$ 
as in defintion \eqref{S}.

We can now state the following corollary.

\begin{corollary}
  Let $L:[a,b]\times\RR^{2n}\to\RR^d$ be of class $C^1$ and satisfy the 
  inequalities \eqref{ineq_L} and \eqref{der_L}.  Let $M\subset
  \mathcal{X}_1$ be an infimizer for $\overline J$.  Given $\zeta\in C^+\bs
  \{0\}$, suppose that there exists a $\zeta$-solution $y_\zeta\in M$.  
  Then for any $v\in W_0^{1,q}([a,b];\RR^n)$
  \begin{equation}
  \label{setELzeta_weak}
  \widehat J'_\zeta(0;M)(v)=S_{(\varphi'_{\zeta,M}(0),\zeta)}(v)=H^+(\zeta)\, .
  \end{equation}
\end{corollary}

The previous equation \eqref{setELzeta_weak} can also be written:
$$
\widehat J'_\zeta(0;M)(v)=S_{\left(-\frac{d\ }{dt}D_pL_\zeta(t,y_\zeta,\dot y_\zeta)
  +D_yL_\zeta(t,y_\zeta,\dot y_\zeta),\zeta\right)}(v)=H^+(\zeta)\, ,
$$
where we have used \eqref{phi'zetaM0bis} and \eqref{weak_sol}.

We assume now that the Lagrangian is coercive in every direction of the 
dual cone, in the sense that there exist $\alpha>0$, $\beta\geq 0$ such that
\begin{equation}
\label{coerc}
L_\zeta(t,y,p)\geq\alpha|p|^q-\beta
\end{equation}
for any $\zeta\in B$, where $B$ is a base of $C^+$, $p,y\in\RR^n$, $t\in[a,b]$.

If $\overline L(t,y,p)$ is convex in $(y,p)$, it is immediate to see that 
$\overline J$ is convex and for every $\zeta\in C^+\bs\{0\}$ 
$L_\zeta(t,y,p)$ is also convex in $(y,p)$.

By the classical results (see \cite{Evans} and \cite{Giaq}), if the coercivity 
inequalities and the convexity condition hold, then for every $\zeta\in C^+
\bs\{0\}$ there exists at least one $\zeta$- solution $y_\zeta\in
\mathcal{X}_1$ solving
\begin{equation}
\label{zeta_pb}
J_\zeta(y_\zeta)=\min_{y\in\mathcal{X}_1} J_\zeta(y)\, .
\end{equation}

The following theorem gives a sufficient condition for $M$ to be an 
infimizer.

\begin{theorem}
\label{trm_Minf}
  Assume that $L$ satisfies the coercivity inequalities (\ref{coerc}) in every 
  direction $\zeta\in B$, where $B$ is a base of $C^+$ 
  and $\overline L(t,y,p)$ is convex in $(y,p)$.  Then the set
  \begin{equation}
  \label{Minf}
  M=\{ y_\zeta\in\mathcal{X}_1\mid\zeta\in B \}\, ,
  \end{equation}
  with $y_\zeta$ $\zeta$-solution, is a $C^+$-solution of
  $\inf_{y\in\mathcal{X}_1} \overline J(y)$.
\end{theorem}

\begin{proof}
  In order to prove that $\inf_{y\in M} \overline J(y)=\inf_{y\in\mathcal{X}_1} 
  \overline J(y)$,
  we only need to prove that for every $y_0\in\mathcal{X}_1$
  \begin{equation}
  \label{Jy0}
  J(y_0)\in\inf_{y\in M} \overline J(y)=\cl\co\bigcup_{\zeta\in C^+
    \bs\{0\}}\overline J(y_\zeta)\, .
  \end{equation}
  Let $\Gamma_M$ be defined as $\Gamma_M=\inf_{y\in M} \overline J(y)$.  
  Let us assume that $J(y_0)\notin\Gamma_M$.   Then by the separation theorem 
  there exist $\xi\in\RR^d$ and $\alpha\in\RR$ such that
  \begin{equation}
  \label{sep}
  \xi\cdot J(y_0)<\alpha<\xi\cdot z\qquad \mbox{for any }z\in\Gamma_M\, .
  \end{equation}
  We want to prove that $\xi\in C^+$.  If it were not, there would be $c\in C$ 
  such that $\xi\cdot c<0$, but this would imply that 
  $\inf_{z\in\Gamma_M}\xi\cdot z=-\infty$ and it would contradict \eqref{sep}.  
  Now, if $\xi\in C^+$, \eqref{sep} means that
  $$
  J_\xi(y_0)<\alpha\leq\inf_{z\in\Gamma_M}\xi\cdot z=J_\xi(y_\xi)
  $$
  and this is a contradiction to the definition of $y_\xi$ as minimizing 
  $J_\xi(\cdot)$.
\end{proof}

If one prefers to work in the space $\F(\RR^d,C)$, then $M$ in \eqref{Minf} must 
be replaced by its convexification.

By recalling Remark \ref{rmk-|zeta|=1}, the infimizer $M$ can be built considering 
only some $\zeta\in C^+\bs\{0\}$.

\begin{example}
  Let us consider the following Lagrangian function $L:[0,1]\times\RR^2\to\RR^2$:
  $$
  L(t,y,p)=\left( \begin{array}{c} p^2+4y^2 \\ tp+p^2 \end{array} \right)
  $$
  and $\mathcal{X}_1=\{ y\in W^{1,2}([0,1];\RR)\mid y(0)=0,\ y(1)=1\}$.  
  We consider three different cones:
  $$
  \begin{aligned}
  C_1 &=\RR^d_+\, ,\\
  C_2 &=\left\{(\rho\cos\theta,\rho\sin\theta)\mid\rho\geq0,\; 
    0\leq\theta\leq\frac{\pi}{3}\right\},\\
  C_3 &=\left\{(\rho\cos\theta,\rho\sin\theta)\mid\rho\geq0,\; 
    -\frac{\pi}{4}\leq\theta\leq\frac{\pi}{2}\right\} .
  \end{aligned}
  $$
  Their dual cones are:
  $$
  \begin{aligned}
  C_1^+ &=\RR^d_+\, ,\\
  C_2^+ &=\left\{(\rho\cos\theta,\rho\sin\theta)\mid\rho\geq0,\; 
    -\frac{\pi}{6}\leq\theta\leq\frac{\pi}{2}\right\},\\
  C_3^+ &=\left\{(\rho\cos\theta,\rho\sin\theta)\mid\rho\geq0,\; 
    0\leq\theta\leq\frac{\pi}{4}\right\} .
  \end{aligned}
  $$
  For $\zeta=(\zeta_1,\zeta_2)$ we have $L_\zeta(t,y,p)=(\zeta_1+\zeta_2)p^2
  +\zeta_2 tp+4\zeta_1 y^2$ 
  and if we consider only $\zeta$ on the line $\zeta_1+\zeta_2=1$ 
  we obtain $L_\zeta(t,y,p)=p^2+(1-\zeta_1)tp+4\zeta_1 y^2$.
  
  We prove that $\overline L$ is convex in $(y,p)$.  It is enough to show that
  for every $t\in[0,1]$, $(y_1,p_1),(y_2,p_2)\in\RR^2$ and $\lambda\in[0,1]$
  $$
  \lambda L(t,y_1,p_1)+(1-\lambda)L(t,y_2,p_2)\in L(t,\lambda y_1+(1-\lambda)y_2,
    \lambda p_1+(1-\lambda)p_2)+C_i
  $$
  with $i=1,2,3$.   Using the convexity of the Lagrangian components, it is easy 
  to see that
  $$
  \lambda L(t,y_1,p_1)+(1-\lambda)L(t,y_2,p_2)- L(t,\lambda y_1+(1-\lambda)y_2,
    \lambda p_1+(1-\lambda)p_2)={a+b\choose a},
  $$
  with $a=\lambda(1-\lambda)(p_1-p_2)^2$, $b=4\lambda(1-\lambda)(y_1-y_2)^2$ 
  nonnegative and this vector is in all the three cones.
  
  The elements $\zeta\in C^+_i$ and such that they belong to the line $\zeta_1
  +\zeta_2=1$ are respectively:
  $$
  \begin{aligned}
  \zeta\in C^+_1\colon&\quad\zeta=\binom{\zeta_1}{1-\zeta_1}\mbox{ with }
    0\leq\zeta_1\leq 1\, , \\
  \zeta\in C^+_2\colon&\quad\zeta=\binom{\zeta_1}{1-\zeta_1}\mbox{ with }
    0\leq\zeta_1\leq\frac{3+\sqrt{3}}{2} \, , \\
  \zeta\in C^+_3\colon&\quad\zeta=\binom{\zeta_1}{1-\zeta_1}\mbox{ with }
    \frac12\leq\zeta_1\leq 1\, .
  \end{aligned}
  $$
  For these $\zeta$ the projection $L_\zeta$ is coercive and the hypotheses of 
  Theorem~\ref{trm_Minf} are met.
  
  The Euler-Lagrange equation is $2\ddot y+1-\zeta_1=8\zeta_1y$.
  For $\zeta_1$ positive the solutions are of the type
  $y(t)=c_1e^{2\sqrt{\zeta_1}t}+c_2e^{-2\sqrt{\zeta_1}t}+
  \frac{1-\zeta_1}{8\zeta_1}$,
  with $c_1,c_2$ constants.   In order to fulfill the boundary conditions, we obtain
  $$
  y_\zeta(t)= \frac{(1-\zeta_1)\left(e^{-2\sqrt{\zeta_1}}-1\right)+8\zeta_1}
      {8\zeta_1\left(e^{2\sqrt{\zeta_1}}-e^{-2\sqrt{\zeta_1}}\right)}
      e^{2\sqrt{\zeta_1}t}
    -\frac{(1-\zeta_1)\left(e^{2\sqrt{\zeta_1}}-1\right)+8\zeta_1}{8\zeta_1
      \left(e^{2\sqrt{\zeta_1}}-e^{-2\sqrt{\zeta_1}}\right)}
      e^{-2\sqrt{\zeta_1}t}+\frac{1-\zeta_1}{8\zeta_1}\, .
  $$
  For $\zeta_1=0$ the solution is $y_\zeta=-\frac14t^2+\frac54t$.
  
  The sets
  $$
  \begin{aligned}
  M_1 &= \left\{ y_\zeta\mid\zeta=\binom{\zeta_1}{\zeta_2}\in C_1^+, 
      \zeta_1+\zeta_2=1\right\} 
    =\left\{ y_\zeta\mid\zeta=\binom{\zeta_1}{1-\zeta_1}, 
      0\leq\zeta_1\leq 1\right\} \\
  M_2 &= \left\{ y_\zeta\mid\zeta=\binom{\zeta_1}{\zeta_2}\in C_2^+, 
      \zeta_1+\zeta_2=1\right\}
    = \left\{ y_\zeta\mid\zeta=\binom{\zeta_1}{1-\zeta_1}, 
      0\leq\zeta_1\leq\frac{3+\sqrt{3}}{2}\right\} \\
  M_3 &= \left\{ y_\zeta\mid\zeta=\binom{\zeta_1}{\zeta_2}\in C_3^+, 
      \zeta_1+\zeta_2=1\right\} 
    = \left\{ y_\zeta\mid\zeta=\binom{\zeta_1}{1-\zeta_1}, 
      \frac12\leq\zeta_1\leq 1\right\} \\
  \end{aligned}
  $$
  are infimizers.
  
  The infima in the three cases are plotted in Figure~\ref{fig2}.
  \begin{figure}
  \begin{center}
  \includegraphics[width=5.2cm]{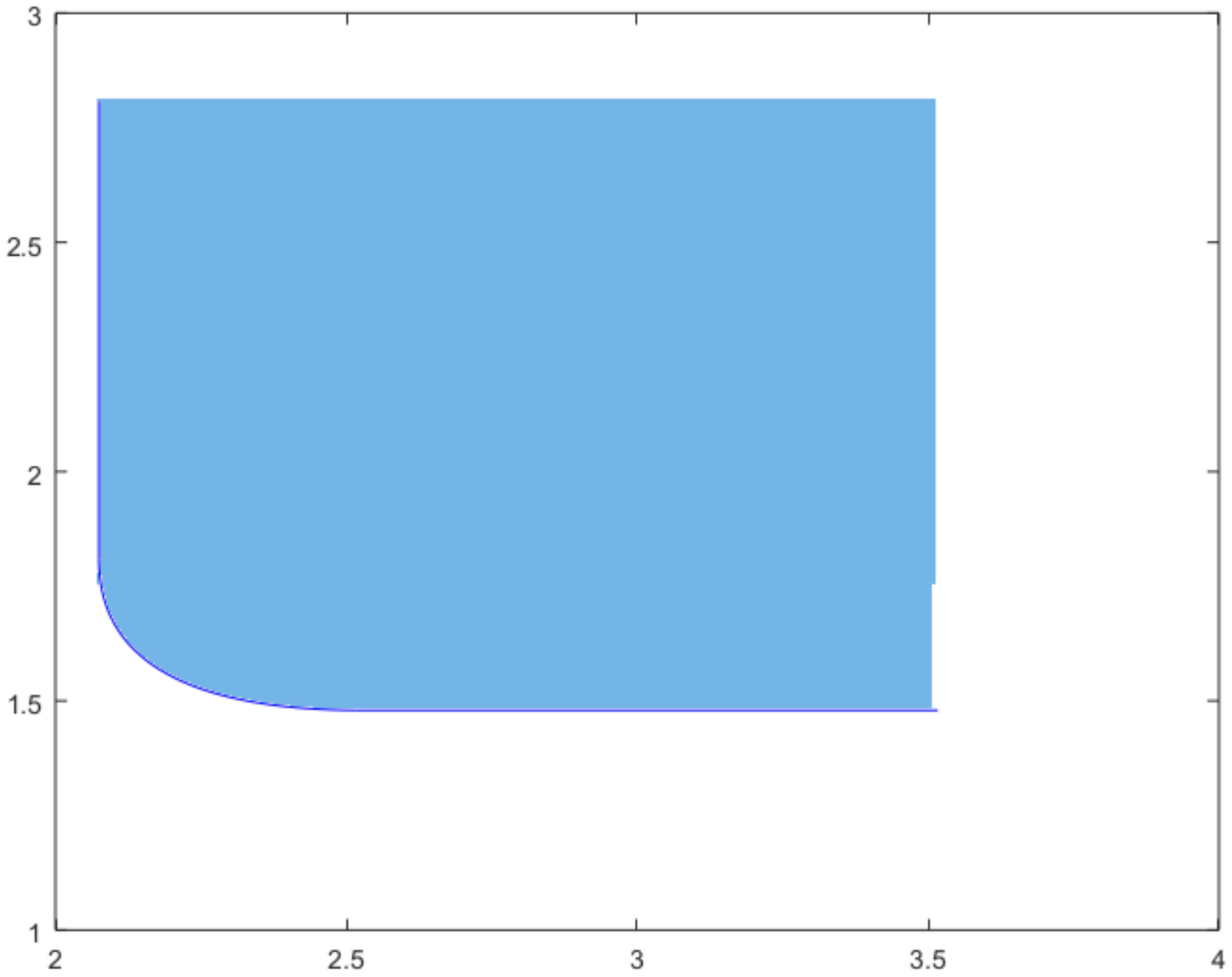}
  \includegraphics[width=5.2cm]{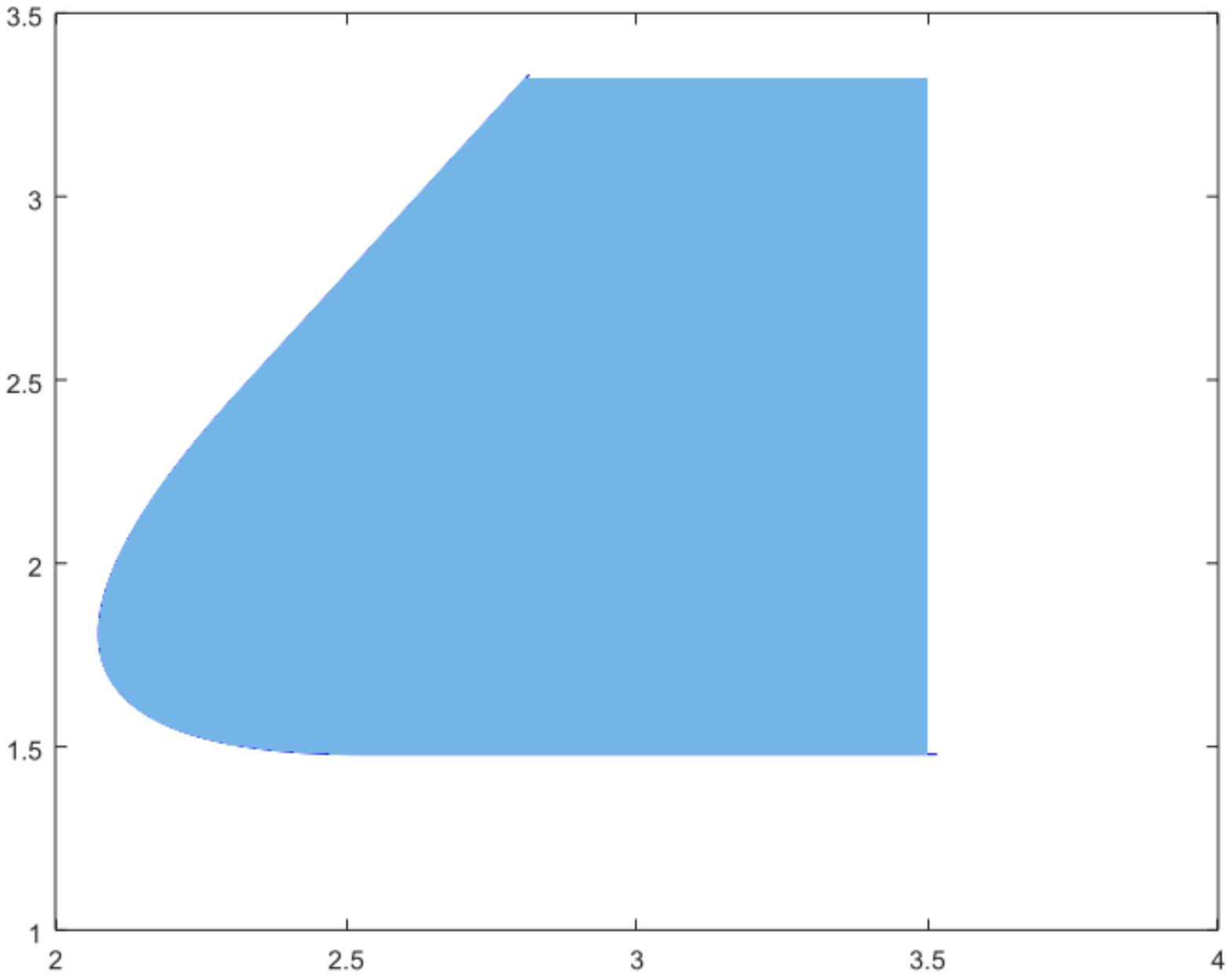}
  \includegraphics[width=5.2cm]{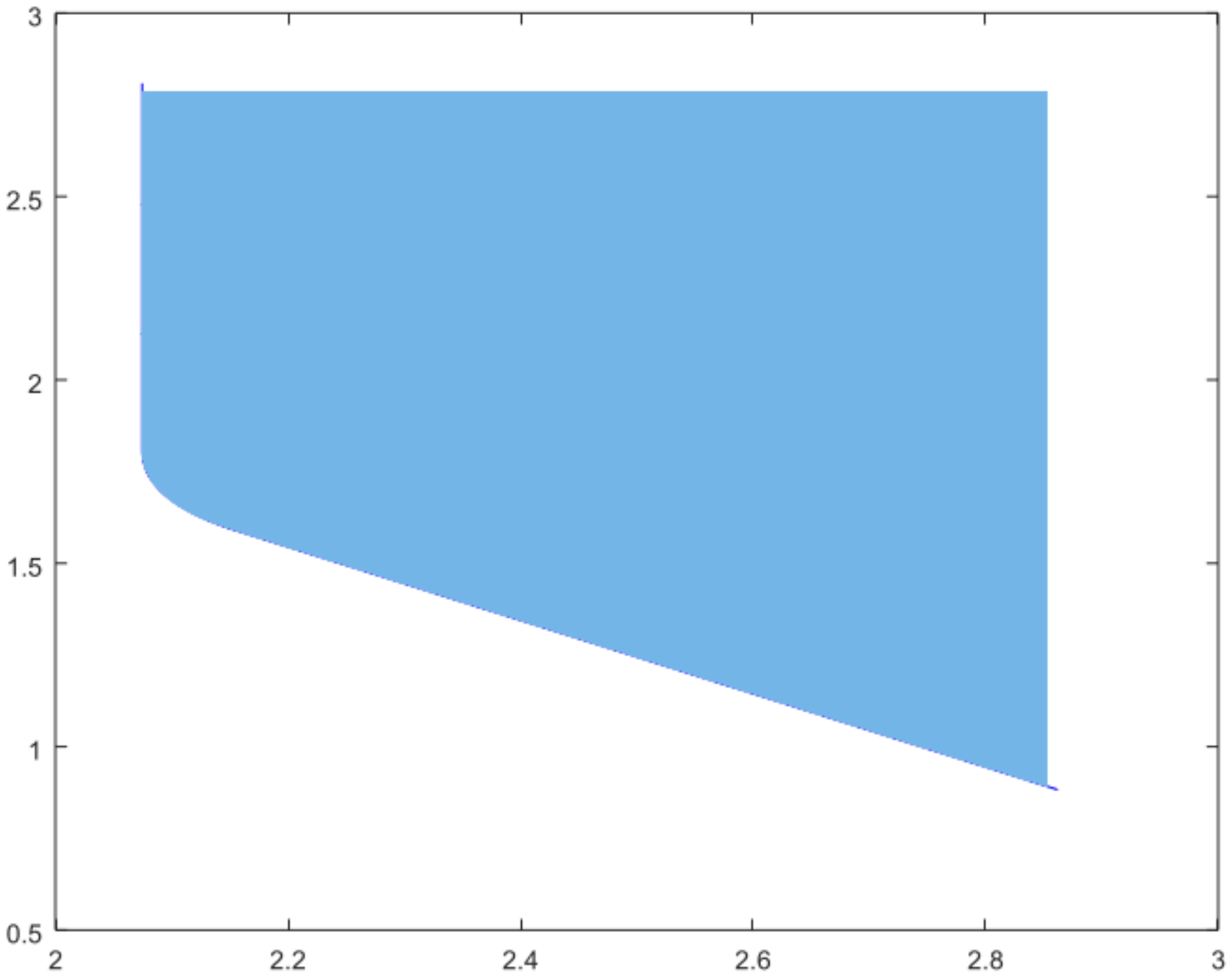}
  \end{center}
  \caption{The infima $\inf\overline J[M_i]$, $i=1,2,3$, that correspond to the three 
  cones are plotted here.}
  \label{fig2}
  \end{figure}
\end{example}

\begin{remark}
  For the Lagrangian of Example~\ref{ex_2}, the coercivity condition is not valid.  
  To fix the ideas, let the cone be $C=\RR^d_+$.   Except for the projection along 
  the horizontal direction, the Lagrangian $L_\zeta$ is not coercive and is not even 
  bounded from below.  All the same, considering the arcs in $W^{1,2}([0,1],\RR)$ 
  and using Poincar\'e inequality, it is possible to prove that
  \begin{equation}
  \label{J(y)2}
  (J(y))_2\geq\frac{1}{2}\|\dot y\|^2_{L^2(0,1)}-K_1\| y\|_{L^2(0,1)}
  \geq \frac{1}{2}\|\dot y\|^2_{L^2(0,1)}-K_2\|\dot y\|_{L^2(0,1)}+K_3\, ,
  \end{equation}
  with $K_1,K_2,K_3$ constants.  In fact, let $y_0$ be any arc in $\mathcal{X}_1$.  
  Then $y-y_0\in W_0^{1,2}([0,1];\RR)$ and there exist $K,K'$ constants such 
  that
  $$
  \begin{aligned}
  \| y\|_{L^2(0,1)} &\leq\| y-y_0\|_{L^2(0,1)}+\| y_0\|_{L^2(0,1)} 
      \leq K\|\dot y-\dot y_0\|_{L^2(0,1)}+\| y_0\|_{L^2(0,1)} \\
    &\leq K\|\dot y\|_{L^2(0,1)}+K'\, .
  \end{aligned}
  $$
  From (\ref{J(y)2}) there exist $\alpha>0$ and $\beta\geq 0$ such that 
  $(J(y))_2\geq\alpha \|\dot y\|^2_{L^2(0,1)}-\beta$.

  With this observation we can see that, even if $L_\zeta$ is not coervive in $p$, 
  the functional $J_\zeta$ is coercive in $y\in W^{1,2}([0,1],\RR)$ and the 
  corresponding $\zeta$-solutions $y_\zeta$ 
  exist.  We can then conclude that $M=\{ y_\zeta(t)\mid\zeta\in\RR^d_+\bs\{0\}\}$
  is an infimizer for the functional $\overline J$.
\end{remark}

%%%%%%%%%%%%%%%%%%%%%%%%%%%%%%%%%%%%%%
\section{Isoperimetric Problems}

In this section we want to consider isoperimetric problems.  See also \cite{Evans} and 
\cite{Giaq}.  In the first part we recall well-known real-valued results and then 
we derive a vector-valued result.

The considered problem is the minimization of the functional
\begin{equation}
\label{overlineJ}
\overline J(y)=\int_a^b\overline L(t,y(t),\dot y(t))\, dt
\end{equation}
on the set
\begin{equation}
\label{X_2}
\mathcal{X}_2= \left\{ y\in W^{1,q}([a,b];\RR^n)\mid y(a)=A,\ y(b)=B,\ \int_a^b 
  G(y(t))\, dt = 0\right\} ,
\end{equation}
where $q>1$ and $G:\RR^n\to\RR^m$ is such that every component verifies
\begin{equation}
\label{nablaGi}
|\nabla G_i(w)|\leq K\left(|w|^{q-1}+1\right)
\end{equation}
for some constant $K$.

The following real-valued existence result of a constrained minimizer holds 
(this result is more general than the one in \cite{Evans}).

\begin{proposition}
  Assume that $L$ satisfies the coercivity inequalities (\ref{coerc}) for any $\zeta\in 
  C^+\bs\{0\}$ such that $|\zeta|=1$ and that $\overline L(t,y,p)$ is convex in 
  $p$. Assume that $G$ satisfies 
  (\ref{nablaGi}).  Moreover, assume that $\mathcal{X}_2$ is not empty.  For any 
  $\zeta\in C^+\bs\{0\}$, there exists a $\zeta$-solution $y_\zeta\in
  \mathcal{X}_2$ satisfying
  \begin{equation}
  \label{y_zeta}
  J_\zeta(y_\zeta)=\inf_{y\in\mathcal{X}_2} J_\zeta(y)\, .
  \end{equation}
\end{proposition}

\begin{proof}
  Let $\{y_k\}_{k\in\NN}\subset\mathcal{X}_2$ be a minimizing sequence:
  $\lim_{k\to+\infty}J_\zeta(y_k)=\inf_{y\in\mathcal{X}_2} J_\zeta(y)=m_0$.
  It is possible to see that the hypothesis of coercivity implies that $\{y_k
  \}_{k\in\NN}$ is bounded in $W^{1,q}([a,b];\RR^n)$.  There exists a 
  subsequence, that we still denote $\{y_k\}_{k\in\NN}$, and a function 
  $y_0\in W^{1,q}([a,b];\RR^n)$ such that $y_k\rightharpoonup y_0$ in 
  $W^{1,q}([a,b];\RR^n)$.
  Since  the functional $J_\zeta$ is weakly lower semicontinuous 
  $J_\zeta(y_0)\leq \lim_{k\to+\infty}J_\zeta(y_k)=m_0$.
  We want to prove that $y_0\in\mathcal{X}_2$.  By Rellich-Kondrachov theorem 
  $W^{1,q}([a,b];\RR^n)$ is compactly embedded in $C([a,b];\RR^n)$.  
  Consequently, both $y_k$ and $G(y_k)$ tend pointwise to $y_0$ and $G(y_0)$, 
  respectively.  From (\ref{nablaGi}), we easily deduce that 
  $|G_i(w)|\leq K'\left(|w|^q+1\right)$, where $K'$ is a positive constant.  Since
  $$
  |G_i(y_k)|\leq K'\left(|y_k|^q+1\right)\leq K''\left(|y_0|^q+1\right),
  $$
  with $K''$ positive constant, by the dominated convergence theorem
  $$
  0=\lim_{k\to +\infty}\int_a^bG_i(y_k(t))dt=\int_a^bG_i(y_0(t))dt\, .
  $$
  Then $y_0\in\mathcal{X}_2$ and coincides with $y_\zeta$.
\end{proof}

Such a minimizer of the real-valued problem (\ref{y_zeta}) is a solution of the 
following Lagrange multiplier problem.

\begin{proposition}
  Let $L:[a,b]\times\RR^{2n}\to\RR^d$ be of class $C^1$ and satisfy 
  \eqref{ineq_L} and \eqref{der_L}.   Assume that $G$ satisfies 
  (\ref{nablaGi}).  Let $y_\zeta\in
  \mathcal{X}_2$ be such that (\ref{y_zeta}) holds.  Suppose 
  that there exist $w_i\in W_0^{1,q}([a,b];\RR^n)$, with $1\leq i\leq m$ such that
  \begin{equation}
  \label{hyp_wi}
  \det\left(\int_a^b\nabla G_i(y_\zeta)\cdot w_j\, dt\right)\neq 0\, .
  \end{equation}
  Then there exists $\lambda=(\lambda_1,\dots,\lambda_d)\in\RR^d$ such that 
  for any $v\in W_0^{1,q}([a,b];\RR^n)$
  \begin{equation}
  \label{Lagr_eq_constr}
  \int_a^b \left[ D_yL_\zeta(t,y_\zeta,\dot y_\zeta)\cdot v+D_pL_\zeta(t,y_\zeta,
    \dot y_\zeta)\cdot\dot v \right] dt - \sum_{i=1}^m \lambda_i \int_a^b 
    \nabla G_i(y_\zeta)\cdot v\, dt=0\, .
  \end{equation}
\end{proposition}

\begin{proof}
  Let us denote by $W$ the matrix $m\times m$, whose elements are 
  $W_{ij}=\left(\int_a^b\nabla G_i(y_\zeta)\cdot w_j\, dt\right)$
  and by $h:\RR\times\RR^m\to\RR^m$ the map
  $$
  h(s,\sigma)=\int_a^b G\left(y_\zeta(t)+sv(t)+\sum_{i=1}^m\sigma_iw_i(t)
    \right) dt\, .
  $$
  Since $h(0,0)=0$ and $\det\frac{\partial h}{\partial\sigma}(0,0)\neq 0$,
  it is possible to apply the implicit function theorem and there exists $\sigma:[0,\delta)
  \to\RR^m$ such that $\sigma(0)=0$, $h(s,\sigma(s))=0$ for all $s\in[0,\delta)$ 
  and
  $$
  \sigma'(0)=-\left(\frac{\partial h}{\partial\sigma}(0,0)\right)^{-1}\frac{\partial h}
    {\partial s}(0,0)=-W^{-1}\frac{\partial h}{\partial s}(0,0)\, ,
  $$
  where the components of the vector $\frac{\partial h}{\partial s}(0,0)$ are
  $\int_a^b\nabla G_i(y_\zeta)\cdot v\, dt$.
  We denote by $w$ the curve $w:[0,\delta)\to W_0^{1,q}([a,b];\RR^n)$ defined 
  as
  \begin{equation}
  \label{w(s)}
  w(s)=sv+\sum_{i=1}^m \sigma_i(s)w_i\, .
  \end{equation}
  The curve $y_\zeta+w(s)$ is in $\mathcal{X}_2$.  By (\ref{y_zeta}) 
  $\left.\frac{d}{ds}J_\zeta(y_\zeta+w(s))\right|_{s=0}=0$ and this gives
  $$
  \begin{aligned}
  \int_a^b &\left[ D_yL_\zeta(t,y_\zeta,\dot y_\zeta)\cdot v+
      D_pL_\zeta(t,y_\zeta,\dot y_\zeta)\cdot\dot v \right] dt \\
    &-\sum_{i,j=1}^m \left(W^{-1}\right)_{ij}\int_a^b [ D_yL_\zeta(t,y_\zeta,
      \dot y_\zeta)\cdot w_i+D_pL_\zeta(t,y_\zeta,\dot y_\zeta)\cdot\dot w_i ] dt 
      \left(\frac{\partial h}{\partial s}(0,0)\right)_j =0\, .
  \end{aligned}
  $$
  The previous equation can be written as
  $$
  \int_a^b \left[ D_yL_\zeta(t,y_\zeta,\dot y_\zeta)\cdot v+
      D_pL_\zeta(t,y_\zeta,\dot y_\zeta)\cdot\dot v \right] dt 
    -\sum_{j=1}^m \lambda_j\int_a^b\nabla G_j(y_\zeta)\cdot v\, dt=0\, ,
  $$
  with
  $$
  \lambda_j=\sum_{i=1}^m \left(W^{-1}\right)_{ij}\int_a^b [ D_yL_\zeta(t,y_\zeta,
      \dot y_\zeta)\cdot w_i+D_pL_\zeta(t,y_\zeta,\dot y_\zeta)\cdot\dot w_i ] dt\, .
  $$
\end{proof}

\begin{remark}
  The arc $y_\zeta$ in the previous proposition is a weak solution of
  \begin{equation}
  \label{EL_eig}
  D_yL_\zeta(t,y,\dot y)-\frac{d}{dt}D_pL_\zeta(t,y,\dot y)
    - \sum_{i=1}^m \lambda_i \nabla G_i(y)=0\, .
  \end{equation}
\end{remark}

\begin{proposition}
  Let $L:[a,b]\times\RR^{2n}\to\RR^d$ be of class $C^1$ and satisfy 
  \eqref{ineq_L} and \eqref{der_L}.   Assume that $G$ satisfies 
  (\ref{nablaGi}).  Let $M\subset
  \mathcal{X}_2$ be an infimizer for $\overline J$.  Given $\zeta\in
  C^+\bs\{0\}$, suppose that there exists a $\zeta$-solution $y_\zeta\in M$.
  Suppose also that there exist $w_i\in W_0^{1,q}([a,b];\RR^n)$, with 
  $1\leq i\leq m$ such that \eqref{hyp_wi} holds.

  For every $v\in W_0^{1,q}([a,b];\RR^n)$,  let $w(s)$ be as in (\ref{w(s)}).  Then 
  there holds
  \begin{equation}
  \label{ELvarphi}
  \begin{aligned}
  \lefteqn{\left.\frac{d}{ds}\varphi_{\zeta,M}(w(s))\right|_{s=0}} \\
    &=\int_a^b \left[ D_yL_\zeta(t,y_\zeta,\dot y_\zeta)\cdot v+D_pL_\zeta(t,y_\zeta,
    \dot y_\zeta)\cdot\dot v \right] dt - \sum_{i=1}^m \lambda_i 
    \int_a^b \nabla G_i(y_\zeta)\cdot v\, dt%\\
    =0\, .
  \end{aligned}
  \end{equation}
\end{proposition}

\begin{proof}
  Since
  $$
  0\leq\frac{1}{s}\left(\varphi_{\zeta,M}(w(s))-\varphi_{\zeta,M}(0)\right)\leq 
    \frac{1}{s}\left(J_\zeta(y_\zeta+w(s))-J_\zeta(y_\zeta)\right),
  $$
  (\ref{ELvarphi}) holds.
\end{proof}

We give now the set valued equation for the isoperimetric problem, 
recalling the definition \eqref{derivative_curve} of derivative along a curve.

\begin{corollary}
\label{prp-weaksol_iso}
  Let $L:[a,b]\times\RR^{2n}\to\RR^d$ be of class $C^1$ and satisfy 
  the inequalities \eqref{ineq_L} and \eqref{der_L}.    Assume that $G$ 
  satisfies (\ref{nablaGi}).    Let $M\subset
  \mathcal{X}_2$ be an infimizer for $\overline J$.  Given $\zeta\in
  C^+\bs\{0\}$, suppose that there exists $y_\zeta\in M$ such that 
  $\varphi_{\zeta,M}(0)=J_\zeta(y_\zeta)$.  
  Suppose that there exist $w_i\in W_0^{1,q}([a,b];\RR^n)$, with $1\leq i\leq m$ 
  such that \eqref{hyp_wi} holds.

  For every $v\in W_0^{1,q}([a,b];\RR^n)$,  let $w(s)$ be as in (\ref{w(s)}).  Then 
  there holds
  \begin{equation}
  \label{set_weak_sol_eig_z}
  \left.\frac{d\ }{ds}\widehat J_\zeta(w(s);M)\right|_{s=0}=
  S_{\left(D_yL_\zeta(t,y_\zeta,\dot y_\zeta)-\frac{d}{dt}D_pL_\zeta(t,
    y_\zeta,\dot y_\zeta)- \sum_{i=1}^m \lambda_i \nabla G_i(y_\zeta),
    \zeta\right)}(v)=H^+(\zeta)\, .
  \end{equation}
\end{corollary}

%%%%%%%%%%%%%%%%%%%%%%%%%%%%%%%%%%%%%%
\section{Shape Optimization of Energy-saving Buildings}

Following the idea in \cite{Marks}, we consider the problem of optimizing the 
shape of a building of volume $V$ and height $h$ by minimizing building costs 
and yearly heating costs.  The plan of the building is defined by two curves 
$y_1(t)$ and $y_2(t)$, with $y_1(t)\geq y_2(t)$.  Here, in the plane $x,y$, the 
positive $y$ direction coincides with south.  The building is assumed to 
extend in the east-west direction from $-a$ to $a$, so that $-a\leq t\leq a$.

The construction cost can be represented by the following 
functional
$$
J_1(y_1,y_2)=\int_{-a}^a \left[ \alpha_1\sqrt{1+\dot y_1^2}+\alpha_2\sqrt{1+
  \dot y_2^2}+\alpha_3 \right] dt\, ,
$$
while the annual heating cost can be represented by
$$
J_2(y_1,y_2)=\int_{-a}^a \left[ \beta_1\sqrt{1+\dot y_1^2}+\beta_2\sqrt{1+
  \dot y_2^2}+\beta_3 - \beta_4\theta(\dot y_1)\sqrt{1+\dot y_1^2} \right] dt\, .
$$
For the detailed explanations of the above functionals and some technical 
motivations see \cite{Marks}.   
Here $\theta$ is the heat gain due to the solar radiation during the heating period, 
but it is chosen in a different way from \cite{Marks}, more precisely 
we choose $\theta(p)=\theta_1+(\theta_1-\theta_2)\frac{p}{\sqrt{1+p^2}}$, 
where $\theta_1,\theta_2$ are the average sums of the total solar radiation on the 
wall facing south ($\theta_1$) and on the walls facing east and west ($\theta_2$).

The curves must fulfill the boundary conditions $y_1(-a)=y_2(-a)=0$ and  $y_1(a)=y_2(a)=0$.  Moreover, there is a constraint because the volume must be $V$:
$$
\int_{-a}^a (y_1(t)-y_2(t))\, dt=\frac{V}{h}\, .
$$

If $J=(J_1,J_2)$, our problem is to minimize the functional $\overline J(y)$ on the set
$$
\mathcal{X}_3=\left\{ y\in W^{1,2}([-a,a];\RR^2)
    \mid y(\pm a)=(0,0),\ \int_{-a}^aG(y)\, dt=0\right\},
$$
where $G:\RR^2\to\RR$ is defined by $G(y_1,y_2)=y_1-y_2-\frac{V}{2ah}$.

The Lagrangian associated to the functional $J$ is $L:[-a,a]\times\RR^2\to\RR^2$ 
defined by
$$
L(t,p_1,p_2)= \left( \alpha_1\sqrt{1+p_1^2}+\alpha_2\sqrt{1+p_2^2}
  +\alpha_3, \beta_1'\sqrt{1+p_1^2}+\beta_2\sqrt{1+p_2^2}
  +\beta_3 - \beta_4'p_1\right),
$$
with $\beta_1'=\beta_1-\beta_4\theta_1$ and $\beta_4'=\beta_4(\theta_1-\theta_2)$.  
Let $C$ be $\RR^2_+$.  If $\zeta=(\zeta_1,1-\zeta_1)$ for $\zeta_1\in[0,1]$, 
the projection of the Lagrangian in the direction $\zeta$ is
$$
\begin{aligned}
L_\zeta(t,p_1,p_2)= &\zeta_1\left( \alpha_1\sqrt{1+p_1^2}+
    \alpha_2\sqrt{1+p_2^2}+\alpha_3\right) \\
  &+(1-\zeta_1) \left(\beta_1'\sqrt{1+p_1^2}+\beta_2\sqrt{1+p_2^2}
    +\beta_3 - \beta_4'p_1\right).
\end{aligned}
$$
In case there are solutions that are $C^2$, we consider the Euler-Lagrange 
equations:
$$
-\frac{d}{dt}D_pL_\zeta(t,\dot y)-\lambda\nabla G(y)=0
$$
and we obtain
\begin{equation}
\label{ddot}
\left\{ \begin{array}{l}
\displaystyle -[\zeta_1\alpha_1+(1-\zeta_1)\beta_1']\frac{\ddot y_1}{\sqrt{(1+
  \dot y_1^2)^3}} -\lambda=0 \\ \\
\displaystyle -[\zeta_1\alpha_2+(1-\zeta_1)\beta_2]\frac{\ddot y_2}{\sqrt{(1+
  \dot y_2^2)^3}}+\lambda=0
\end{array} \right.
\end{equation}
From these ordinary differential equations it is possible to see that if $\lambda$ is 
zero, the solutions $y_1,y_2$ would be straight lines, which is not an interesting 
solution to the problem proposed.  We look for solutions with $\lambda\neq 0$. 
For simplicity we write $A_{1,\zeta}=\zeta_1\alpha_1+(1-\zeta_1)\beta_1'$ and 
$A_{2,\zeta}=\zeta_1\alpha_2+(1-\zeta_1)\beta_2$.  
By integrating (\ref{ddot}), one obtains
$$
\left\{ \begin{array}{l}
\displaystyle A_{1,\zeta}\frac{\dot y_1(t)}{\sqrt{1+\dot y_1^2(t)}}
  =K_1-\lambda t \\ \\

\displaystyle A_{2,\zeta}\frac{\dot y_2(t)}{\sqrt{1+\dot y_2^2(t)}}
  =K_2+\lambda t
\end{array} \right.
$$
where $K_1,K_2$ are constants, and this gives
$$
\left\{ \begin{array}{l}
\displaystyle \dot y_1(t) = \frac{K_1-\lambda t}{\sqrt{A_{1,\zeta}^2-(K_1-
  \lambda t)^2}}\\ \\
\displaystyle \dot y_2(t) = \frac{K_2+\lambda t}{\sqrt{A_{2,\zeta}^2-(K_2+
  \lambda t)^2}}
\end{array} \right.
$$
A second integration, using the boundary condition in $-a$, gives
$$
\left\{ \begin{array}{l}
\displaystyle y_1(t) = \frac{1}{\lambda} \left( \sqrt{A_{1,\zeta}^2-(K_1-
  \lambda t)^2} -\sqrt{A_{1,\zeta}^2-(K_1+\lambda a)^2} \right) \\ \\
\displaystyle y_2(t) = -\frac{1}{\lambda} \left( \sqrt{A_{2,\zeta}^2-(K_2+
  \lambda t)^2} -\sqrt{A_{2,\zeta}^2-(K_2-\lambda a)^2} \right)
\end{array} \right.
$$
Using the boundary condition in $a$, we obtain $K_1=K_2=0$.  So the 
components of the solution $y_\zeta$ are
\begin{equation}
\label{yzeta_ex}
\left\{ \begin{array}{l}
\displaystyle y_1(t) = \frac{1}{\lambda} \left( \sqrt{A_{1,\zeta}^2-
  \lambda^2t^2} -\sqrt{A_{1,\zeta}^2-\lambda^2a^2} \right) \\ \\
\displaystyle y_2(t) = -\frac{1}{\lambda} \left( \sqrt{A_{2,\zeta}^2-
  \lambda^2t^2} -\sqrt{A_{2,\zeta}^2-\lambda^2a^2} \right)
\end{array} \right.
\end{equation}
See Figure~\ref{fig3} for an example of such a solution.
\begin{figure}
\begin{center}
\includegraphics[width=6cm]{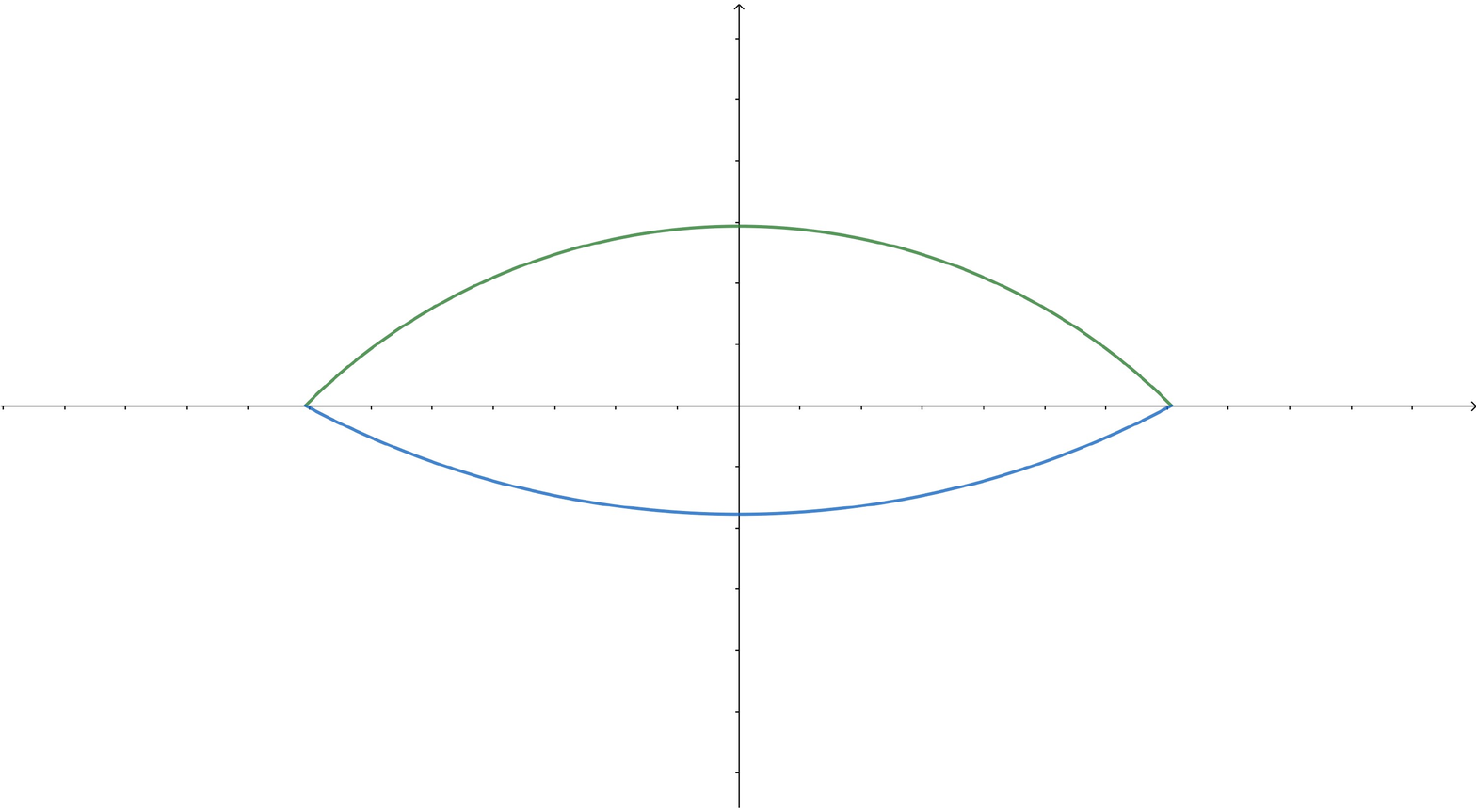}
\end{center}
\caption{The solution in \eqref{yzeta_ex} gives a building with this shape.}
\label{fig3}
\end{figure}
The constraint can be written
\begin{equation}
\label{constraint}
\begin{aligned}
\frac{1}{\lambda} &\left( \frac{A_{1,\zeta}^2}{\lambda}\arctan\frac{\lambda a}
    {\sqrt{A_{1,\zeta}^2-\lambda^2a^2}}-a\sqrt{A_{1,\zeta}^2-\lambda^2a^2}\right.\\  
  &\left.+\frac{A_{2,\zeta}^2}{\lambda}\arctan\frac{\lambda a}
    {\sqrt{A_{2,\zeta}^2-\lambda^2a^2}}-a\sqrt{A_{2,\zeta}^2-\lambda^2a^2} 
    \right) =\frac{V}{h}\, .
\end{aligned}
\end{equation}
Since the limit of the left-hand side of (\ref{constraint}) for $\lambda$ tending to 
zero is zero, a solution $\lambda>0$ exists for example if $A_{1,\zeta}<A_{2,\zeta}$ 
and
$$
\frac{a}{A_{1,\zeta}} \left( \frac{A_{1,\zeta} a\pi}{2}+
  \frac{A_{2,\zeta}^2a}{A_{1,\zeta}}\arctan
  \frac{A_{1,\zeta}}{\sqrt{A_{2,\zeta}^2-A_{1,\zeta}^2}}-
  a\sqrt{A_{2,\zeta}^2-A_{1,\zeta}^2} \right)\geq\frac{V}{h}\, .
$$

Reasoning as in \cite[section 3.5]{Troutman}, it is possible to conclude that the solution 
(\ref{yzeta_ex}) is the minimizer in the direction $\zeta$.  In fact, the Lagrangian 
$L_\zeta(t,p_1,p_2)-\lambda G(y_1,y_2)$ 
is convex in $(y_1,y_2,p_1,p_2)$ and strictly convex in $(p_1,p_2)$ for any 
$\lambda$.  Then the solution (\ref{yzeta_ex}) is the unique minimizer of class $C^1$ 
of the functional
$\int_{-a}^a L_\zeta(t,\dot y_1,\dot y_2)\, dt$ on $\mathcal{X}_3$.

Equation \eqref{set_weak_sol_eig_z} is in this case
$$
S_{\left(-\frac{d}{dt}D_pL_\zeta(t,y,\dot y)- \lambda \nabla G(y),
    \zeta\right)}(v)=H^+(\zeta)
$$
and we can propose the set of all the arcs $y_\zeta$ for $\zeta=(\zeta_1,1-\zeta_1)$, 
$\zeta_1\in[0,1]$, as a candidate infimizer.

%%%%%%%%%%%%%%%%%%%%%%%%%%%%%%%%%%%%%
\section{Conclusions}

A new approach to multi-criteria calculus of variations problems based on set 
optimization methods is proposed. In particular, the problem is extended to a 
problem with an objective function mapping into a complete lattice of sets. The 
corresponding solution concept is twofold: in addition to more traditional minimal 
solutions (also called efficient or non-dominated in multi-criteria/vector 
optimization), infimizer sets are searched for. Thus, optimality conditions are 
needed which characterize infimizers.

In this paper, conditions are presented which characterize infimizers consisting of arcs which in turn solve weighted sum scalarizations of the original problem and thus satisfy the known optimality conditions for scalarized problems. This works well in the convex case. The extension to the non-convex case is very desirable and could be addressed using nonlinear scalarizations of translative or oriented distance type (see for 
example \cite{Z_scal}, \cite{GJN_SetOpt} and \cite{HJNV_scal}).

\medskip
{\bf Acknowledgement.}

The authors are indebted to Andreas Hamel for fruitful suggestions and discussions.

%%%%%%%%%%%%%%%%%%%%%%%%%%%%%%%%%%%%%%


\begin{thebibliography}{99}

\bibitem{Aubin_Frank} \textsc{J.P.~Aubin, H.~Frankowska,} \emph{Set-Valued 
  Analysis,} Birkh\"auser, Boston-Basel-Berlin 1990.

\bibitem{Aumann} \textsc{R.J.~Aumann,} \emph{Integrals of set-valued functions,} 
  J. Math. Anal. Appl.  \textbf{12} (1965), 1--12.

\bibitem{Evans} \textsc{L.C.~Evans}, \emph{Partial Differential Equations}, 
  American Mathematical Society (1998).

\bibitem{Giaq} \textsc{M.~Giaquinta, S.~Hildebrandt,} \emph{Calculus of 
  Variations I,} Springer-Verlag Berlin Heidelberg (2004).

\bibitem{GJN_SetOpt} \textsc{C.~Guti\'errez, B.~Jim\'enez, V.~Novo,} 
  \emph{Nonlinear scalarizations of set optimization problems with set orderings}, 
  in: A.H.~Hamel,  F.~Heyde, A.~L\"ohne, B.~Rudloff, C.~Schrage (Eds.), Set 
  Optimization and Applications -- The State of the Art, Springer 2015, 43--63.

\bibitem{SetOptSurvey} \textsc{A.H.~Hamel, F.~Heyde, A.~L\"ohne, B.~Rudloff, 
  C.~Schrage,} \emph{Set optimization -- A rather short introduction}, in: 
  A.H.~Hamel,     F.~Heyde, A.~L\"ohne, B.~Rudloff, C.~Schrage (Eds.), Set 
  Optimization and Applications -- The State of the Art, Springer 2015, 65--141.

\bibitem{inf-transl} \textsc{A.H.~Hamel, F.~Heyde, D.~Visetti,} \emph{The 
  inf-translation for solving set minimization problems} (2020),
  \texttt{https://arxiv.org/abs/2007.05588}.

\bibitem{HS} \textsc{A.H.~Hamel, C.~Schrage,} \emph{Directional derivatives, 
  subdifferentials and optimality conditions for set-valued convex functions,} 
  Pacific Journal of Optimization \textbf{10} 4 (2014), 667--689.

\bibitem{HV} \textsc{A.H.~Hamel, D.~Visetti}, \emph{The value function approach 
  and Hopf-Lax formula for multiobjective costs via set optimization,} Journal of 
  Mathematical Analysis and Applications \textbf{483} (2020).
  
\bibitem{HL} \textsc{F.~Heyde, A.~L\"ohne,} \emph{Solution concepts in vector   
  optimization: a fresh look at an old story,} Optimization \textbf{60} 12 (2011), 
  1421--1440.

\bibitem{HJNV_scal} \textsc{L.~Huerga, B.~Jim\'enez, V.~Novo, A.~V\'\i lchez}, 
  \emph{Six set scalarizations based on the oriented distance: continuity, convexity 
  and application to convex set optimization}, Mathematical Methods of Operations 
  Research (2021).

\bibitem{KCS} \textsc{A.~Konak, D.W.~Coit, A.E.~Smith,} \emph{Multi-objective   
  optimization using genetic algorithms: A tutorial}, Reliability Engineering and System 
  Safety \textbf{91} (2006), 992--1007.

\bibitem{Marks} \textsc{W.~Marks,} \emph{Multicriteria optimisation of shape of 
  energy-saving buildings,} Building and Environment \textbf{32} 4 (1997), 
  331--339.

\bibitem{MarlerArora} \textsc{R.T.~Marler, J.S.~Arora,} \emph{Survey of 
  multi-objective optimization methods for engineering,} Structural and 
  Multidisciplinary Optimization \textbf{26} 6 (2004), 369--395.

\bibitem{Rangaiah} \textsc{G.P. Rangaiah} (Ed.), \emph{Multi-Objective 
  Optimization. Techniques and Applications in Chemical Engineering,} 
  Advances in Process Systems Engineering \textbf{5}, 
  World Scientific 2016.

\bibitem{rc} \textsc{R.W.~Reid, S.J.~Citron,} \emph{On noninferior performance 
  index vectors,} Journal of Optimization Theory and Applications \textbf{7} 1 
  (1971), 11--28.

\bibitem{Troutman} \textsc{J.L.~Troutman,} \emph{Variational Calculus with 
  Elementary Convexity,} Springer-Verlag New York (1983).

\bibitem{yl} \textsc{P.L.~Yu, G.~Leitmann,} \emph{Nondominated decisions 
  and cone convexity in dynamic multicriteria decision problems,} Journal of 
  Optimization Theory and Applications \textbf{14} 5 (1974), 573--584.

\bibitem{Z_scal} \textsc{A.~Zaffaroni,} \emph{Degrees of Efficiency and Degrees 
  of Minimality,} SIAM Journal on Control and Optimization \textbf{42} 3 (2003), 
  1071--1086. 

\end{thebibliography}
\end{document}